\documentclass[a4paper,10pt,twoside]{amsart}
\usepackage{amssymb}
\usepackage[all]{xy}
\newcommand{\R}{\mathbb{R}}

\newcommand{\Z}{\mathbb{Z}}
\newcommand{\Q}{\mathbb{Q}}

\newcommand{\liealg}[1]{\mathfrak{#1}}
\newcommand{\skal}[2]{\langle{#1},{#2}\rangle}
\newcommand{\Mod}[2]{{#1}/\!\raisebox{-.575ex}{\ensuremath{#2}}}
\newtheorem{theorem}{Theorem}[section]
\newtheorem{lemma}[theorem]{Lemma}
\newtheorem{proposition}[theorem]{Proposition}
\newtheorem{corollary}[theorem]{Corollary}

\theoremstyle{definition}
\newtheorem{definition}[theorem]{Definition}
\newtheorem{example}[theorem]{Example}

\theoremstyle{remark}
\newtheorem{remark}[theorem]{Remark}

\numberwithin{equation}{section}

\begin{document}
%
\title{Riemannian Foliations and the Topology of Lorentzian Manifolds}
%
%
\author{Kordian L\"{a}rz}
\address{Humboldt-Universit\"{a}t Berlin, Institut f\"{u}r Mathematik R. 1.305, Rudower Chaussee 25, 12489 Berlin}
\email{laerz@math.hu-berlin.de}
\thanks{{\sc Acknowledgment:} I am grateful to Helga Baum for many useful discussions on the topic.}
%
%
\subjclass[2000]{53C29, 53C50}
\keywords{Lorentzian geometry, holonomy group, Riemannian foliation, pp-wave}
%
%
%
\begin{abstract}
	A parallel lightlike vector field on a Lorentzian manifold $X$ naturally defines a foliation $\mathcal{F}$ of codimension one.
	If either all leaves of $\mathcal{F}$ are compact or $X$ itself is compact admitting a compact leaf and the (transverse) Ricci curvature
	is non-negative then a Bochner type argument implies that the first Betti number of $X$ is bounded by $1 \leq b_{1} \leq \dim X$ if $X$ is
	compact and $0 \leq b_{1} \leq \dim X -1$ otherwise. We show that these bounds are optimal and depending on the holonomy of $X$ we obtain
	further results. Finally, we classify the holonomy representations for those $X$ admitting a compact leaf with finite fundamental group.
\end{abstract}
\maketitle
%
\section{The Class of Decent Spacetimes}
Let $(X,g^{L})$ be a Lorentzian manifold and $\nabla^{L}$ its Levi-Civita connection.\footnote{All manifolds are assumed to be connected,
												Lorentzian manifolds are assumed to be orientable.}
Suppose $(X,g^{L})$ admits a $\nabla^{L}$-parallel lightlike subbundle $\Xi \subset TX$ of rank one, i.e.,
$\nabla^{L}\Gamma(U \subset X,\Xi) \subset \Gamma(U,\Xi)$. We write $\Xi^{\perp} \subset TX$ for its orthogonal complement. Thus,
$\Xi^{\perp} \supset \Xi$ has codimension one. Being a $\nabla^{L}$-parallel subbundle, $\Xi$ and therefore $\Xi^{\perp}$ induce a foliation
$\mathcal{X}$ of dimension one and a foliation $\mathcal{X}^{\perp}$ of codimension one on $X$. Consider the vector bundle
$\mathcal{S}:= \mbox{Coker}(\Xi \hookrightarrow \Xi^{\perp})$. We have an induced metric $h^{\mathcal{S}}$ and an induced connection
$\nabla^{\mathcal{S}}$ on $\mathcal{S}$. Moreover, $h^{\mathcal{S}}$ has Riemannian signature and
$\nabla^{\mathcal{S}}h^{\mathcal{S}}=0$. We call $(\mathcal{S},h^{\mathcal{S}},\nabla^{\mathcal{S}})$ the (canonical) screen bundle of $(X,g)$. Given a
non-canonical splitting $s$ of the exact sequence
\begin{equation*}
	\xymatrix{ 0 \ar[r] & \Xi \ar[r] & \Xi^{\perp} \ar[r] & \mathcal{S} \ar@/_/[l]_{s} \ar[r] & 0}
\end{equation*}
we define $S:=s(\mathcal{S})$ and call it a (non-canonical) realization of $\mathcal{S}$ in $TX$. The connection $\nabla^{L}$ on $X$ induces a connection
on $S$ given by $\nabla^{S} := pr_{S} \circ \nabla^{L}|_{S}$.\par
The canonical bundle morphism $S \stackrel{F}{\rightarrow} \mathcal{S}$ is easily shown to be a vector bundle isomorphism such that
$\nabla^{S} = F^{*}\nabla^{\mathcal{S}}$ and $g|_{S \times S}=F^{*}h^{\mathcal{S}}$, i.e., $Hol(S,\nabla^{S})=Hol(\mathcal{S},\nabla^{\mathcal{S}})$.
Since $\Xi \subset S^{\perp}$ and $S^{\perp} \subset TX$ has signature $(1,1)$ the light cone in $S^{\perp}_{p}$ is the union of two lines
one of which is given by $\Xi_{p}$ and we derive
\begin{corollary}\label{complementary-distrib}
	Given a realization $S \subset TX$ of the screen bundle of $(X,g^{L})$ there is a uniquely defined lightlike subbundle
	$\Theta \subset TX$ of rank one with the following property: If $V \in \Gamma(U \subset X,\Xi)$ then there exists a unique section
	$Z \in \Gamma(U \subset X,\Theta)$ such that $g^{L}(V,Z)=1$.\qed
\end{corollary}
Using locally future pointing sections as well as Cor. \ref{complementary-distrib} and a partition of unity we conclude that the following are equivalent.
\begin{itemize}
	\item
	$\Xi$ admits a nowhere vanishing section,
	\item
	$(X,g^{L})$ is time-orientable,
	\item
	$\mathcal{X}^{\perp}$ is transversely orientable.
\end{itemize}
Since $\Xi$ is $\nabla^{L}$-parallel any global section is recurrent.\footnote{We say $V \in \Gamma(U,TX)$ is {\em recurrent} if
										$\nabla^{L}_{\cdot}{V} = \alpha_{U}(\cdot)V$ for some 1-form
										$\alpha_{U} \in \Gamma(U,T^{*}X)$.}
\begin{definition}
	Let $(X,g^{L})$ be a Lorentzian manifold and $V \in \Gamma(X,TX)$ a global nowhere vanishing lightlike vector field.
	We say $(X,g^{L},V)$ is an
	\begin{enumerate}
		\item
		almost decent spacetime if $\nabla^{L}_{\cdot}{V} = \alpha(\cdot)V$ for some 1-form $\alpha \in \Gamma(X,T^{*}X)$.
		\item
		decent spacetime if it is almost decent and $\alpha|_{\Xi^{\perp}}=0$.\qed
	\end{enumerate}
\end{definition}
For an almost decent spacetime we always assume that $V \in \Gamma(X,\Xi)$ is future pointing. Next, we characterize the class of almost decent
spacetimes in the class of Lorentzian manifolds. If $(X,g^{L})$ is an arbitrary Lorentzian manifold let $\liealg{hol}_{p}(X,g^{L})$ be its holonomy
algebra at $p \in X$. Then $\liealg{hol}_{p}(X,g^{L})$ has the Borel-Lichn\'erowicz property, i.e., there is an orthogonal decomposition
$T_{p}X=E_{0} \oplus \ldots \oplus E_{\ell}$ into non-degenerate $\liealg{hol}_{p}(X,g^{L})$-invariant subspaces and a corresponding decomposition
$\liealg{hol}_{p}(X,g^{L})=\liealg{h}_{1}\oplus \ldots \oplus \liealg{h}_{\ell}$ into commuting ideals such that each
$\liealg{h}_{j} \subset \liealg{so}(E_{j},g^{L}|_{E_{j}})$ acts weakly irreducibly on $E_{j}$ and trivially on $E_{i}$ for $i \neq j$. Using
\cite{MR1836778} we derive three possible cases:
\begin{enumerate}
	\item
	$E_{0}=0$ or $g^{L}|_{E_{0}}$ is positive definite and $\liealg{h}_{i}$ acts irreducibly for $i \geq 1$. In this case, we may assume that
	$g^{L}|_{E_{j}}$ is positive definite for $j \geq 2$. Hence, $\liealg{h}_{j}$ acts as an irreducible Riemannian holonomy representation for
	$j \geq 2$ and $\liealg{h}_{1} = \liealg{so}(1,n+1)$.
	\item
	$E_{0} \neq 0$ and $g^{L}|_{E_{0}}$ is negative definite or of Lorentzian signature. Thus, $g^{L}|_{E_{j}}$ is positive definite and
	$\liealg{h}_{j}$ acts as an irreducible Riemannian holonomy representation for $j \geq 1$.
	\item
	$E_{0}=0$ or $g^{L}|_{E_{0}}$ is positive definite, $\liealg{h}_{j}$ acts as an irreducible Riemannian holonomy representation for $j \geq 2$
	and $\liealg{h}_{1} \subset \liealg{so}(1,n+1)$ is weakly irreducible but not irreducible. In this case, $\liealg{h}_{1}$ leaves a degenerate
	subspace $W$ invariant.
\end{enumerate}
In the first case, $\liealg{hol}_{p}(X,g^{L})$ does not leave any lightlike line invariant. Hence, $(X,g^{L})$ is not almost decent. There is no
general statement for the second case. However, if $(X,g^{L})$ is given by the last case then it is almost decent if it is time-orientable. Let us
explain this fact. First, we have an $\liealg{h}_{1}$-invariant line $W \cap W^{\perp}$. If $v$ is a lightlike vector in $T_{p}X$ spanning the invariant
line then $Hol^{0}(X,g^{L}) \subset Stab(\R \cdot v) \subset SO_{0}(T_{p}X)$. If we identify $T_{p}X$ with $\R^{1,n+1}$ then it can be shown that
$Stab(\R \cdot v) \cong (\R^{*} \times SO(n)) \ltimes \R^{n}$. If we choose a basis $(v,e_{1},\ldots,e_{n},z)$ of $\R^{n+2}$
satisfying $g(e_{i},e_{j})=\delta_{ij}$, $g(v,z)=1$ and $g(v,v)=g(z,z)=0$ then the Lie algebra $(\R \oplus \liealg{so}(n)) \ltimes \R^{n}$ of
$(\R^{*} \times SO(n)) \ltimes \R^{n}$ is given by
\begin{equation*}
	\left\lbrace \begin{pmatrix} a & w^{T} & 0\\ 0 & A & -w\\ 0 & 0 & -a \end{pmatrix}
								: a \in \R,~A \in \liealg{so}(n),~w \in \R^{n} \right\rbrace.
\end{equation*}
\begin{lemma}
	For a Lorentzian manifold $(X,g^{L})$ whose Borel-Lichn\'erowicz decomposition is given by
	\begin{equation*}
		T_{p}X=E_{0} \oplus \ldots \oplus E_{\ell} \quad \text{and}
						\quad \liealg{hol}_{p}(X,g^{L}) = \liealg{h}_{1}\oplus \ldots \oplus \liealg{h}_{\ell}
	\end{equation*}
	with $E_{i}$ positive definite for $i \neq 1$ and $\liealg{h}_{1} \neq \liealg{so}(1,n+1)$ let $W \cap W^{\perp}$ be an isotropic
	$\liealg{h}_{1}$-invariant subspace. Then $W \cap W^{\perp} \subset E_{1}$ is invariant under the action of the full holonomy group
	$Hol(X,g^{L})$.
\end{lemma}
\begin{proof}
	The idea is to apply $Hol(X,g^{L}) \subset \mbox{Norm}_{O(1,\dim X -1)}(Hol^{0}(X,g^{L}))$.
	For $v \in W \cap W^{\perp}$ we have $\liealg{h}_{1}\cdot v \in \R \cdot v$. Let $H_{i}$ be the connected Lie subgroup of
	$Hol^{0}(X,g^{L})$ whose Lie algebra is $\liealg{h}_{i}$.
	For any $h \in Hol^{0}(X,g^{L})$ we have $h(v)=\alpha_{h} \cdot v$ since $H_{i}$ acts trivially on $E_{1}$ for $i \neq 1$. Therefore,
	$(g^{-1}hg)(g^{-1}v) = \alpha_{h} \cdot g^{-1}(v)$ for $g \in \mbox{Norm}_{O(1,\dim X -1)}(Hol^{0}(X,g^{L}))$ and $h \in Hol^{0}(X,g^{L})$.
	For $0 \leq i \leq \ell$ let $\tilde{v}_{i} \in E_{i}$ such that $g^{-1}(v) = \tilde{v}_{0} + \ldots + \tilde{v}_{\ell}$.
	Using $H_{i} \subset g^{-1}Hol^{0}(X,g^{L})g$ we derive
	$\R \cdot g^{-1}(v) \ni h \cdot g^{-1}(v) =\tilde{v}_{0} +\ldots +\tilde{v}_{i-1} +h\tilde{v}_{i} +\tilde{v}_{i+1} +\ldots +\tilde{v}_{\ell}$
	for all $h \in H_{i}$. Therefore, $h \tilde{v}_{i} \in \R \cdot \tilde{v}_{i}$ and for $i \geq 2$ we conclude $\tilde{v}_{i}=0$ since $H_{i}$
	acts irreducibly. Hence, $g^{-1}(v)=\tilde{v}_{0}+\tilde{v}_{1}$.\par
	On the other hand, we have $\R \cdot g^{-1}(v) \ni h \cdot g^{-1}(v) = \tilde{v}_{0}+ h\tilde{v}_{1}$ for all $h \in H_{1}$. Hence,
	$\tilde{v}_{1} \in \R \cdot v$ since $H_{1}$ acts weakly irreducibly and reducibly on $E_{1}$. If $\tilde{v}_{0} \neq 0$ we derive the
	contradiction $0 = \skal{g^{-1}(v)}{g^{-1}(v)} =\skal{\tilde{v}_{0}}{\tilde{v}_{0}}+2\skal{\tilde{v}_{0}}{\tilde{v}_{1}}
							+\skal{\tilde{v}_{1}}{\tilde{v}_{1}} = \skal{\tilde{v}_{0}}{\tilde{v}_{0}} \neq 0$
	since $E_{0}$ is definite. Therefore, $g^{-1}(v) \in \R \cdot v$ and $Hol(X,g^{L}) \cdot v \in \R \cdot v$.
\end{proof}
We conclude that $W \cap W^{\perp}$ corresponds to a $\nabla^{L}$-parallel lightlike subbundle $\Xi \subset TX$ of rank one. If $(X,g^{L})$ is
time-orientable\footnote{If $(X,g)$ is not time-orientable we may consider its 2-fold time-orientation cover. The global nowhere vanishing section
					$V \in \Gamma(X,\Xi)$ is recurrent but not necessarily parallel even if $\liealg{h}_{1}$ annihilates a vector.
					In fact, we derive a class in $H^{1}(X,\R)$ induced by
					$\pi_{1}(X) \twoheadrightarrow \Mod{Hol_{p}(\nabla^{L})}{Hol^{0}_{p}(\nabla^{L})} \rightarrow \R$ where the last
					morphism is induced by $pr_{\Xi_{p}} \circ Hol_{p}(\nabla^{L})|_{\Xi_{p}}$.}
we have a global section $V \in \Gamma(X,\Xi)$, i.e., $(X,g^{L},V)$ is almost decent.
\section{A Lorentzian - Riemannian Dictionary}
Let $(X,g^{L},V)$ be an almost decent spacetime and $S$ a realization of the screen bundle. Using Cor. \ref{complementary-distrib} we fix
$Z \in \Gamma(X,\Theta)$ and define the following Riemannian metric on $X$.
\begin{equation*}
	g^{R}(A,B) := \begin{cases}
				1 &\text{if}~A=B=V~\text{or}~A=B=Z,\\
				g^{L}(A,B) &\text{if}~A,B \in S,\\
				0 &\text{otherwise}.
			\end{cases}
\end{equation*}
Given a choice for $S$ and $V$ we say $g^{R}$ is the $(V,S)$-metric associated to $g^{L}$. We have $T\mathcal{X}^{\perp} = \Xi \oplus S$ and
$\Theta = (T\mathcal{X}^{\perp})^{\perp_{g^{R}}}$ where $(T\mathcal{X}^{\perp})^{\perp_{g^{R}}}$ is the normal bundle of $T\mathcal{X}^{\perp} \subset TX$
w.r.t. $g^{R}$.
\begin{definition}
	Let $(X,\mathcal{F})$ be a foliated manifold and $\Gamma(U,T\mathcal{F})$ the vector fields on $U \subset X$ tangent to $\mathcal{F}$.
	\begin{enumerate}
		\item
		A Riemannian metric $g^{R}$ on $X$ is bundle-like w.r.t. $\mathcal{F}$ if $(L_{V}g)(Y_{1},Y_{2})=0$ for any open
		subset $U \subset X$ and all $V \in \Gamma(U,T\mathcal{F})$, $Y_{\cdot} \in \Gamma(U,T\mathcal{F}^{\perp_{g^{R}}})$.
		\item
		We say $(X,\mathcal{F})$ is transversely parallelizable if there exists a global frame
		$(\bar{Y}_{1},\ldots,\bar{Y}_{\operatorname{co dim}\mathcal{F}})$ for $\Mod{TX}{T\mathcal{F}}$ and
		global sections $Y_{i} \in \Gamma(X,TX)$ such that $[Y_{i},T\mathcal{F}] \in T\mathcal{F}$ and
		$\bar{Y}_{i}=pr_{\Mod{TX}{T\mathcal{F}}}(Y_{i})$ for all $1 \leq i \leq \operatorname{co dim}\mathcal{F}$.\qed
	\end{enumerate}
\end{definition}
\begin{lemma}\label{parallel-bundle-like}
	Let $(X,g^{L},V)$ be an almost decent spacetime. For any realization of the screen bundle $S$ the following are equivalent.
	\begin{enumerate}
		\item
		The $(V,S)$-metric $g^{R}$ is bundle-like w.r.t. $\mathcal{X}^{\perp}$ and $(X,\mathcal{X}^{\perp})$ is transversely parallelizable,
		\item
		the 1-form $g^{L}(V,\cdot)$ defining $\mathcal{X}^{\perp}$ is closed and
		\item
		$(X,g^{L})$ is decent, i.e., $\alpha|_{\Xi^{\perp}}=0$.
	\end{enumerate}
\end{lemma}
\begin{proof}
	Suppose $\alpha|_{\Xi^{\perp}}=0$. Let $V \in \Gamma(X,\Xi)$ and fix $Z \in \Gamma(X,\Theta)$. We have to show $(L_{W}g^{R})(Z,Z)=0$
	for all $W \in \Gamma(U,\Xi^{\perp})$. Using $g^{R}(\cdot,Z) = g^{L}(\cdot,V)$ we derive
	$g^{R}(\nabla^{L}_{W}{Z},Z)=g^{L}(\nabla^{L}_{W}{Z},V)=-g^{L}(Z,\nabla^{L}_{W}{V})=-\alpha(W)$ and
	\begin{align*}
		(L_{W}g^{R})(Z,Z) &= \underbrace{W(g^{R}(Z,Z))}_{=0} - 2g^{R}([W,Z],Z)
				= 2g^{R}(\underbrace{\nabla^{L}_{Z}{W}}_{\in \Xi^{\perp}},Z)-2g^{R}(\nabla^{L}_{W}{Z},Z)\\
				&= 2\alpha(W).
	\end{align*}
	Thus, $g^{R}$ is bundle-like w.r.t. $\mathcal{X}^{\perp}$. Moreover, $pr_{Z}([W,Z]):= g^{R}([W,Z],Z)Z = -\alpha(W)Z$ and $Z$ is globally
	defined, i.e., the foliation $\mathcal{X}^{\perp}$ is transversely parallelizable.
	For the last statement we compute
	\begin{align*}
		d(g^{L}(V,\cdot))(W,Z) &= g^{L}(\nabla^{L}_{W}{V},Z) - g^{L}(\nabla^{L}_{Z}{V},W)\\
					&= \alpha(W)g^{L}(V,Z) -\alpha(Z)g^{L}(V,W) = \alpha(W).
	\end{align*}
	For the converse we follow these equations backwards.
\end{proof}
\begin{lemma}\label{bundle-like-restrictions}
	Let $(X,g^{L},V)$ be an almost decent spacetime and $\mathcal{L}^{\perp}$ a leaf of $\Xi^{\perp}$. Let $S$ be any realization of the screen
	bundle.
	\begin{enumerate}
		\item
		The restriction $g^{R}|_{\mathcal{L}^{\perp}}$ of the $(V,S)$-metric is bundle-like w.r.t. the foliation
		$(\mathcal{L}^{\perp},\mathcal{X}|_{\mathcal{L}^{\perp}})$.\footnote{This fact seems to be well known and the first reference I
											could find is \cite{MR1677118}.}
		\item
		The $(V,S)$-metric is bundle-like w.r.t. the foliation $(X,\mathcal{X})$ if and only if
		$\alpha(V)=0$\footnote{The integral curves of $V$ are $g^{L}$-geodesics if and only if $\alpha(V)=0$ since
						$\nabla^{L}_{\cdot}V = \alpha(\cdot)V$.}
		and $[V,Z] \in \Gamma(X,\Xi)$.
	\end{enumerate}
\end{lemma}
\begin{proof}
	Since $g^{R}|_{S \times S} = g^{L}|_{S \times S}$ we have for any $Y_{1},Y_{2} \in \Gamma(U,S)$
	\begin{align*}
		(L_{V}g^{R})(Y_{1},Y_{2}) &= V(g^{L}(Y_{1},Y_{2})) -g^{L}([V,Y_{1}],Y_{2}) -g^{L}([V,Y_{2}],Y_{1})\\
					&= g^{L}(\nabla^{L}_{V}{Y_{1}},Y_{2}) + g^{L}(Y_{1},\nabla^{L}_{V}{Y_{2}})
							-g^{L}([V,Y_{1}],Y_{2}) -g^{L}([V,Y_{2}],Y_{1})\\
					&= g^{L}(\nabla^{L}_{Y_{1}}{V},Y_{2}) + g^{L}(Y_{1},\nabla^{L}_{Y_{2}}{V}) = 0
	\end{align*}
	For the second statement we need to show $(L_{V}g^{R})(Y_{1},Y_{2})=0$ for any $Y_{1},Y_{2} \in \Gamma(U,S \oplus \Theta)$. If $Y_{1}=Z$ and
	$Y_{2} \in \Gamma(U,S)$ we derive
	\begin{align*}
		(L_{V}g^{R})(Z,Y_{2}) &= V(g^{R}(Z,Y_{2})) -g^{R}([V,Z],Y_{2}) -g^{R}([V,Y_{2}],Z)\\
					&= -g^{L}([V,Z],Y_{2}) -g^{L}(\underbrace{[V,Y_{2}]}_{\in \Xi^{\perp}},V) = -g^{L}([V,Z],Y_{2}).
	\end{align*}
	For $Y_{1}=Y_{2}=Z$ we have $(L_{V}g^{R})(Z,Z) = V(g^{R}(Z,Z)) -2g^{R}([V,Z],Z) = 2\alpha(V)$. Since $g^{L}(\nabla^{L}_{V}Z,V) = -\alpha(V)$ we
	conclude $[V,Z] \in \Gamma(X,\Xi)$ if $\alpha(V)=0$ and $[V,Z] \in \Gamma(X,\Xi \oplus \Theta)$.
\end{proof}
For a foliated manifold $(X,\mathcal{F})$ let $\Mod{X}{\mathcal{F}}$ be its set of leaves and
\begin{equation*}
	\pi: X \rightarrow \Mod{X}{\mathcal{F}}, \qquad p \mapsto (\text{leaf through}~p).
\end{equation*}
For Riemannian foliations this map has been studied in \cite{MR0142130}, \cite{MR0120592}, \cite{MR637184} and
\cite{MR932463}. Given the results in \cite{MR0370617} and Lemma \ref{parallel-bundle-like} we have
\begin{corollary}[Conlon \cite{MR0370617}]\label{all-leaves-closed}
	For a decent spacetime $(X,g^{L},V)$ all leaves of $(X,\mathcal{X}^{\perp})$ have trivial leaf-holonomy. Suppose there is a realization
	of the screen bundle such that $Z$ is complete and let $\mathcal{L}^{\perp}$ be a leaf of $\mathcal{X}^{\perp}$.
	\begin{enumerate}
		\item
		If there is no leaf of $\mathcal{X}^{\perp}$ which is closed in $X$ then each leaf is dense in $X$.
		\item
		We have $\tilde{X} = \tilde{\mathcal{L}}^{\perp} \times \R$ where $\tilde{X}$, $\tilde{\mathcal{L}^{\perp}}$ denote the universal covers
		of $X$, $\mathcal{L}^{\perp}$.
		\item
		If there is a closed leaf then $X \rightarrow \Mod{X}{\mathcal{X}^{\perp}}$ is a smooth fiber bundle and
		$\Mod{X}{\mathcal{X}^{\perp}} \in \{\R,S^{1}\}$.
		\item
		The inclusion $\mathcal{L}^{\perp} \rightarrow X$ induces a monomorphism $\pi_{1}(\mathcal{L}^{\perp}) \rightarrow \pi_{1}(X)$ onto
		a normal subgroup. If $X$ is compact then $\Mod{\pi_{1}(X)}{\pi_{1}(\mathcal{L}^{\perp})}=\Z^{r}$ for some $r \geq 1$ and
		$r=1$ if and only if $\mathcal{L}^{\perp}$ is closed in $X$.\qed
	\end{enumerate}
\end{corollary}
A spacetime $(X,g^{L})$ is said to be distinguishing at $p \in X$ if for any neighborhood $U \ni p$ there is a neighborhood $V \subset U$ such that
$p \in V$ and any (piecewise smooth) causal curve $\gamma: [a,b] \rightarrow X$ with $\gamma(a)=p$ and $\gamma(b) \in V$ is contained in $V$. We say
$(X,g^{L})$ is a distinguishing spacetime if it is distinguishing for all $p \in X$. On the causality ladder (cf. \cite{MR2436235}) we have
\begin{equation*}
	\text{strongly causal} \Rightarrow \text{distinguishing} \Rightarrow \text{causal}.
\end{equation*}
\begin{proposition}\label{strong-causal-foliation}
	Let $(X,g^{L},V)$ be an almost decent spacetime.
	\begin{enumerate}
		\item
		If $(X,g^{L})$ is causal then the leaves of the foliated manifolds $(X,\mathcal{X})$ and
		$(\mathcal{L}^{\perp},\mathcal{X}|_{\mathcal{L}^{\perp}})$ have trivial leaf holonomy. Moreover, $X$ is not compact.
		\item
		If $(X,g^{L})$ is distinguishing at $p \in X$ then the leaf $\mathcal{L}^{\perp}$ of $\mathcal{X}^{\perp}$ through $p$ is not
		compact.
		\item
		If $(X,g^{L})$ is distinguishing then each leaf of $\mathcal{X}$ is a closed subset in $X$ and each leaf of
		$\mathcal{X}|_{\mathcal{L}^{\perp}}$ is a closed subset in $\mathcal{L}^{\perp}$.
	\end{enumerate}
\end{proposition}
\begin{proof}
	Any curve in a leaf $\mathcal{L}$ of $\mathcal{X}$ is lightlike, i.e., $\pi_{1}(\mathcal{L})=0$ since $(X,g^{L})$ is causal.\par
	Suppose $\mathcal{L}^{\perp}$ is compact. There is a bundle-like Riemannian metric on the compact foliated manifold
	$(\mathcal{L}^{\perp},\mathcal{X}|_{\mathcal{L}^{\perp}})$. Consider the leaf $\mathcal{L} \subset \mathcal{L}^{\perp}$ of
	$\mathcal{X}|_{\mathcal{L}^{\perp}}$ through $p$. If $\mathcal{L}$ is closed in $\mathcal{L}^{\perp}$ then it is compact, i.e., we have a
	closed lightlike curve through $p$. In this case $(X,g^{L})$ would not be causal at $p$. On the other hand, if $\mathcal{L}$ is not closed
	in $\mathcal{L}^{\perp}$ then $\bar{\mathcal{L}} \subset \mathcal{L}^{\perp}$ is diffeomorphic to a torus in $\mathcal{L}^{\perp}$
	by \cite{MR744829} and $(X,g^{L})$ is not distinguishing at $p$.\par
	Let $\mathcal{L} \subset \mathcal{L}^{\perp}$ be a leaf of $\mathcal{X}$. Suppose we have $q \in \bar{\mathcal{L}} \setminus \mathcal{L}$
	where the closure is taken w.r.t. $X$. For any $X$-open neighborhood $U \ni q$ we can find $p_{U} \in \mathcal{L} \cap U$. In particular,
	we may choose $U$ to be a coordinate neighborhood ball such that $\bar{U} \subset \tilde{U}$ where $\tilde{U}$ is a Walker coordinate
	neighborhood, i.e., $g^{L}= 2dxdz+ u_{\alpha}dy^{\alpha}dz + fdz^{2} + g_{\alpha\beta}dy^{\alpha}dy^{\beta}$ in $\tilde{U}$ and
	$V \in \mbox{span}\{\partial_{x}\}$. In these coordinates we have $p_{U}=(x_{0},y^{1}_{0}.\ldots,y^{n}_{0},z_{0})$
	and a curve segment $\gamma:[0,b] \rightarrow U$ with $t \mapsto (t+x_{0},y^{1}_{0},\ldots,y^{n}_{0},z_{0})$. Thus, $\dot{\gamma} \in \Xi$
	implies $\gamma([0,b]) \subset \mathcal{L}$ and since $\bar{U} \subset \tilde{U}$ we may assume $\gamma(b) \notin U$. Finally, we can find
	$\tilde{p}_{U} \in \mathcal{L} \cap U$ such that $\tilde{p}_{U} \notin \{(\cdot,y^{1}_{0},\ldots,y^{n}_{0},z_{0})\}$ and since
	$\mathcal{L}$ is connected there is a curve $\tilde{\gamma}$ in $\mathcal{L}$ connecting $\gamma(b)$ and $\tilde{p}_{U}$.
	Therefore, we have a (piecewise smooth) lightlike curve from $p_{U}$ to $\tilde{p}_{U}$ which leaves $U$ and $(X,g^{L})$ is not
	distinguishing at $q$.
	Hence, $\mathcal{L}$ is closed in $X$ and being the preimage of a closed set under $\mathcal{L}^{\perp} \rightarrow X$ it is closed
	in $\mathcal{L}^{\perp}$.
\end{proof}
Consider a leaf $\mathcal{L}^{\perp}$ in a distinguishing almost decent spacetime. By Prop. \ref{strong-causal-foliation} all leaves of the
foliation $\mathcal{X}|_{\mathcal{L}^{\perp}}$ are closed with vanishing fundamental group. Hence,
$\mathcal{L}^{\perp}\rightarrow \Mod{\mathcal{L}^{\perp}}{\mathcal{X}}$ is a submersion if $\Mod{\mathcal{L}^{\perp}}{\mathcal{X}}$ is Hausdorff
\cite{MR1453120}. Since $g^{R}|_{\mathcal{L}^{\perp}}$ is bundle-like a sufficient condition is that all segments $\gamma$ of horizontal geodesics, i.e.,
$\dot{\gamma}(0) \in S|_{\mathcal{L}^{\perp}}$, can be indefinitely extended (cf. \cite{MR0142130}). However, for
$Y_{\cdot} \in \Gamma(U,S)$ the Koszul formula implies $g^{R}(\nabla^{R}_{Y_{1}}{Y_{2}},Y_{3}) = g^{L}(\nabla^{L}_{Y_{1}}{Y_{2}},Y_{3})$. Hence, a curve
$\gamma$ tangent to $S$ is a horizontal geodesic if $\nabla^{S}_{\dot{\gamma}}{\dot{\gamma}}=0$.
\begin{example}\label{bazaikin-example}
	Let $(M,\tilde{g})$ be a simply connected compact Riemannian manifold and $f \in C^{\infty}(M)$. For $\varepsilon >0$ and $L \in \{\R, S^{1}\}$
	define $X:= S^{1} \times L \times M$ and
	\begin{equation*}
		g^{L}_{\varepsilon} := 2dxdz + \varepsilon fdz^{2} + \tilde{g}
	\end{equation*}
	where $dx$ and $dz$ are the standard coordinate 1-forms on $S^{1} \times L$. If $f \in C^{\infty}(M)$ is suitable then $(X,g^{L}_{\varepsilon})$
	is weakly irreducible where $\partial_{x}$ is $\nabla^{g^{L}_{\varepsilon}}$-parallel. Moreover, the leaves of $(X,\mathcal{X}^{\perp})$
	are compact and the universal cover of $(X,g^{L}_{\varepsilon})$ is globally hyperbolic if $\varepsilon$ is sufficiently small.
\end{example}
\begin{proof}
	Each leaf of $\mathcal{X}^{\perp}$ is diffeomorphic to $S^{1} \times M$ and the universal cover of $X$ is given by $\R^{2} \times M$.
	The pullback of $g^{L}_{\varepsilon}$ to $\R^{2} \times M$ is of the form $2dxdz + \varepsilon fdz^{2} + g$ where $x$ and $z$ are
	the standard coordinates on $\R^{2}$. Bazaikin has shown in \cite[Thm. 2]{bazaikin-2009} that this metric is globally hyperbolic if
	$\varepsilon$ is sufficiently small.
\end{proof}
For a Riemannian foliation $(X,\mathcal{F})$ with a bundle-like metric $g^{R}$ the {\em transverse Levi-Civita connection} $\nabla^{T}$ on
$(T\mathcal{F})^{\perp}$ is given by
\begin{equation*}
	\nabla^{T}_{X}{Y} = \begin{cases}
				\pi_{(T\mathcal{F})^{\perp}}(\nabla^{g^{R}}_{X}{Y}) & X \in (T\mathcal{F})^{\perp},\\
				\pi_{(T\mathcal{F})^{\perp}}([X,Y]) & X \in T\mathcal{F},
	                    \end{cases}
\end{equation*}
where $Y \in \Gamma(U,(T\mathcal{F})^{\perp})$. Consider the foliation $(\mathcal{L}^{\perp},\mathcal{X}|_{\mathcal{L}^{\perp}})$ with the $(V,S)$-metric
$g^{R}|_{\mathcal{L}^{\perp}}$. Any local section $\tilde{V} \in \Gamma(U,\Xi)$ is given by $\tilde{V}=fV$. Hence,
\begin{equation*}
	\nabla^{T}_{\tilde{V}}{Y} = \pi_{(T\mathcal{F})^{\perp}}(\nabla^{L}_{\tilde{V}}{Y})
						-\pi_{(T\mathcal{F})^{\perp}}(\underbrace{\nabla^{L}_{Y}{fV}}_{\in \Xi})
				= \pi_{(T\mathcal{F})^{\perp}}(\nabla^{L}_{\tilde{V}}{Y}) = \nabla^{S}_{\tilde{V}}{Y}
\end{equation*}
and using $g^{R}(\nabla^{R}_{Y_{1}}{Y_{2}},Y_{3}) = g^{L}(\nabla^{L}_{Y_{1}}{Y_{2}},Y_{3})$ we conclude
\begin{corollary}\label{transverse-levi-civita}
	Let $(X,g^{L},V)$ be an almost decent spacetime and $\mathcal{L}^{\perp}$ a leaf of $\mathcal{X}^{\perp}$. For any realization $S$ of the
	screen bundle the transverse Levi-Civita connection coincides with $\nabla^{S}|_{\mathcal{L}^{\perp}}$.\qed
\end{corollary}
\begin{definition}
	Let $(X,g^{L},V)$ be an almost decent spacetime. If $S$ is a realization of the screen bundle we say
	$(X,g^{L},V,S)$ is
	\begin{itemize}
		\item
		almost horizontal if $\alpha(Y)=g^{L}(Z,\nabla^{L}_{V}{Y})$ or equivalently $[V,Y] \in S$ for any local section $Y \in \Gamma(U,S)$,
		\item
		horizontal if it is almost horizontal and decent.\qed
	\end{itemize}
\end{definition}
Hence, $\nabla^{L}_{V}{Y} \in \Gamma(U,S)$ for any section $Y \in \Gamma(U,S)$ if $(X,g^{L},V,S)$ is horizontal. In particular, $d(g^{L}(Z,\cdot))(V,\cdot)|_{\Xi^{\perp}}=-g^{L}(Z,[V,\cdot])|_{\Xi^{\perp}}=0$ if and only if $(X,g^{L},V,S)$ is almost horizontal.
\begin{lemma}\label{horizontal-isometric-flow}
	Let $(X,g^{L},V)$ be an almost decent spacetime. If $S$ is a realization of the screen bundle then
	\begin{enumerate}
		\item
		$(X,g^{L},V,S)$ is almost horizontal if and only if for any leaf $\mathcal{L}^{\perp}$ of $\mathcal{X}^{\perp}$ the restriction of
		$g^{R}|_{\mathcal{L}^{\perp}}$ of the $(V,S)$-metric defines the structure of an isometric Riemannian flow on
		$(\mathcal{L}^{\perp},\mathcal{X}|_{\mathcal{L}^{\perp}})$, i.e., $L_{V}g^{R}(W_{1},W_{2})=0$ for all $W_{1},W_{2} \in \Xi^{\perp}$
		and $V$ is a $g^{R}|_{\mathcal{L}^{\perp}}$-Killing vector field of constant length.
		\item
		The $(V,S)$-metric is bundle-like w.r.t. the foliation $(X,\mathcal{X})$ and $\alpha|_{S}=0$ if and only if $(X,g^{L},V,S)$ is horizontal.
		\item
		The $(V,S)$-metric $g^{R}$ defines the structure of an isometric Riemannian flow on $(X,\mathcal{X})$ and $\alpha|_{S}=0$ if and only
		if $(X,g^{L},V,S)$ is horizontal and $\alpha =0$.
	\end{enumerate}
\end{lemma}
\begin{proof}
	Lemma \ref{bundle-like-restrictions} implies $L_{V}g^{R}(W_{1},W_{2})=0$ for all $W_{1},W_{2} \in S$ and the first equivalence follows from
	\begin{equation*}
		L_{V}g^{R}(V,W_{2}) =V(g^{R}(V,W_{2})) -g^{R}([V,V],W_{2}) -g^{R}([V,W_{2}],V) = -g^{L}([V,W_{2}],Z).
	\end{equation*}
	If $(X,g^{L},V,S)$ is horizontal we have a local orthonormal frame $(Y_{1},\ldots,Y_{\dim S})$ for $S$
	such that $[V,Y_{i}] \in S$. Thus, $g^{L}(\nabla^{L}_{V}{Z},Y_{i})=-g^{L}(Z,\nabla^{L}_{V}{Y_{i}})=0$ and
	$pr_{S}([V,Z])= pr_{S}(\nabla^{L}_{V}{Z})$ imply $[V,Z] \in \Gamma(X,\Xi)$. This implies the second equivalence by
	Lemma \ref{bundle-like-restrictions}. For the last statement we consider $L_{V}g^{R}(V,Z)=-g^{R}([V,Z],V)=\alpha(Z)$.%
\end{proof}
\begin{proposition}
	Let $(X,g^{L},V,S)$ be a horizontal spacetime. If $f \in C^{\infty}(X)$ is $(X,\mathcal{X})$-basic, i.e., $V(f)=0$, then $(X,g^{f})$ is a
	horizontal spacetime where the transverse conformal change $g^{f}$ of $g^{L}$ by $f$ is defined by
	\begin{equation*}
		g^{f} := \begin{cases}
				g^{f}|_{S \times S} = e^{f}g^{L}|_{S \times S},\\
				g^{f}(V,V) = g^{f}(Z,Z) = g^{f}(V,S) = g^{f}(Z,S) =0,\\
				g^{f}(V,Z) =1.
		         \end{cases}
	\end{equation*}
\end{proposition}
\begin{proof}
	First, we show that $\nabla^{f}_{\cdot}{V}=\nabla^{L}_{\cdot}{V}$. Let $(V,Y_{1},\ldots,Y_{\dim S},Z)$ be a local frame of $(X,g^{L})$ where
	$(Y_{\alpha})_{\alpha}$ is a $g^{L}$-orthonormal frame for $S$. The Koszul formula and $\alpha|_{\Xi^{\perp}}=0$ imply for
	$U_{1},U_{2} \in \{V,Y_{\cdot},Z\}$
	\begin{equation*}
		2g^{f}(\nabla^{f}_{U_{1}}{V},U_{2}) = Vg^{f}(U_{1},U_{2}) + g^{f}([U_{1},V],U_{2}) - g^{f}([V,U_{2}],U_{1}).
	\end{equation*}
	If $U_{1}=U_{2}=Z$ we derive $2g^{f}(\nabla^{f}_{Z}V,Z)=2g^{L}(\nabla^{L}_{Z}V,Z)=2\alpha(Z)$. If $U_{1},U_{2} \in \Xi^{\perp}$ we have
	$[V,U_{\cdot}] \in \Xi^{\perp}$. Hence,
	\begin{align*}
		2g^{f}(\nabla^{f}_{U_{1}}{V},U_{2}) &= V(e^{f})g^{L}(U_{1},U_{2})\\
						&\qquad + e^{f}(Vg^{L}(U_{1},U_{2}) + g^{L}([U_{1},V],U_{2}) - g^{L}([V,U_{2}],U_{1}))\\
						&= V(e^{f})g^{L}(U_{1},U_{2}) = 0
	\end{align*}
	since $f$ is $(X,\mathcal{X})$-basic. If $U_{1}=Z$ and $U_{2} \in S$ we conclude
	$2g^{f}(\nabla^{f}_{U_{1}}{V},U_{2})= e^{f}g^{L}([Z,V],U_{2}) -g^{L}([V,U_{2}],Z) = (e^{f}-1)g^{L}(Z,\nabla^{L}_{V}{U_{2}})=0$
	since $(X,g^{L})$ is horizontal. The case $U_{1} \in S$ and $U_{2}=Z$ is similar. On the other hand, $U_{1}=V$ and $U_{2}=Z$ implies
	$2g^{f}(\nabla^{f}_{U_{1}}{V},U_{2})= -g^{L}([V,Z],V)=\alpha(V)=0$. Finally,
	\begin{align*}
		2g^{f}(\nabla^{f}_{V}{Y_{\cdot}},Z) &= g^{f}(\underbrace{[V,Y_{\cdot}]}_{\in S},Z) -g^{f}([V,Z],Y_{\cdot}) -g^{f}([Y_{\cdot},Z],V)\\
				&=e^{f}g^{L}(Z,\underbrace{\nabla^{L}_{V}{Y_{\cdot}}}_{\in S}) +g^{L}(Z,\nabla^{L}_{Y_{\cdot}}{V}) = 0.
	\end{align*}
	Hence, $(X,g^{f})$ is horizontal.
\end{proof}
If $(X,g^{L})$ is a Walker coordinate neighborhood of the form $g^{L}=2dxdz + u_{\alpha}dy^{\alpha}dz + hdz^{2} + g_{\alpha\beta}dy^{\alpha}dy^{\beta}$
and we choose $V:=\partial_{x}$ and $Z:=\partial_{z}-\frac{1}{2}h\partial_{x}$ then the transverse conformal change is given by
$g^{f}=2dxdz + u_{\alpha}dy^{\alpha}dz + hdz^{2}+e^{f}g_{\alpha\beta}dy^{\alpha}dy^{\beta}$.\par
If $(X,g^{L},V,S)$ is horizontal then $[V,Z] \in \Gamma(X,\Xi)$, i.e., $V$ and $Z$ induce a 2-dimensional foliation on $X$. The $(V,S)$-metric $g^{R}$
is bundle-like w.r.t. this foliation if $(L_{Z}g^{R})|_{S \times S}=0$. If $(X,g^{L})$ is a Walker coordinate neighborhood as above this condition
corresponds to $\partial_{z}g_{\alpha\beta}=0$.
\begin{corollary}\label{horizontal-to-orbibundle}
	Let $(X,g^{L},V,S)$ be an almost horizontal spacetime and $\mathcal{L}^{\perp}$ a leaf of $\mathcal{X}^{\perp}$. If all leaves of
	$\mathcal{X}|_{\mathcal{L}^{\perp}}$ are compact then the projection
	$\mathcal{L}^{\perp} \rightarrow \Mod{\mathcal{L}^{\perp}}{\mathcal{X}|_{\mathcal{L}^{\perp}}}$ is a principal $S^{1}$-orbibundle over
	$\Mod{\mathcal{L}^{\perp}}{\mathcal{X}|_{\mathcal{L}^{\perp}}}$ for which $S|_{\mathcal{L}^{\perp}}$ defines a connection whose connection
	1-form is $g^{L}(Z,\cdot)$.
\end{corollary}
\begin{proof}
	Since $(X,g^{L},V,S)$ is almost horizontal $(\mathcal{L}^{\perp},\mathcal{X}|_{\mathcal{L}^{\perp}},g^{R}|_{\mathcal{L}^{\perp}})$ is an
	isometric flow and the statement follows from \cite[Prop. 3.7]{MR932463} and a theorem of Wadsley.
	See \cite[Thm. 6.3.8]{MR2382957} or \cite[Prop. 3]{MR2482083} for the details.
\end{proof}
The following examples of horizontal spacetimes show that the leaves of $\mathcal{X}$ and $\mathcal{X}^{\perp}$ are not necessarily closed.
\begin{example}
	Let $(M:=S^{1} \times S^{1},g)$ be the flat torus and $a \in \R \setminus \Q$. Write $\partial_{x}$, $\partial_{y}$ for the standard coordinate
	vector fields on $M$ and define $\eta:=g(\partial_{x}-\frac{1}{a}\partial_{y},\cdot)$. The trivial $S^{1}$-bundle $S^{1} \times T^{2}$ admits a
	weakly irreducible horizontal Lorentzian metric $g^{L}$ such that the leaves of $\mathcal{X}$ are the fibers of the bundle. Moreover, all leaves
	of $\mathcal{X}^{\perp}$ are dense in $S^{1} \times T^{2}$.
\end{example}
\begin{proof}
	The construction of the metric is a special case of \cite[Prop. 3.1]{laerz-2008-a}. There the foliation $\mathcal{X}^{\perp}$ is defined by
	$\pi^{*}\eta$. If $\eta:=g(\partial_{x}-\frac{1}{a}\partial_{y},\cdot)$ then $\mbox{Ker}\eta = \mbox{span}\{\partial_{x}+a\partial_{y}\}$.
	Hence, the leaves of $\pi^{*}\eta$ are dense in $S^{1} \times T^{2}$.
\end{proof}
\begin{example}
	Let $X:= T^{2} \times S^{1}$ where $T^{2}:=S^{1}\times S^{1}$ and $a \in \R \setminus \Q$. Write $\partial_{x},\partial_{y}$ for the standard
	coordinate vector fields of $T^{2}$ and $\partial_{z}$ for the last standard coordinate vector field in $T^{2} \times S^{1}$. For
	$f \in C^{\infty}(T^{2})$ define $g^{L}_{f}$ by
	\begin{align*}
		&g^{L}_{f}(\partial_{x}+a\partial_{y},\partial_{x}+a\partial_{y})=g^{L}_{f}(\partial_{x}+a\partial_{y},\partial_{y})
									=g^{L}_{f}(\partial_{y},\partial_{z}) = 0,\\
		&g^{L}_{f}(\partial_{x}+a\partial_{y},\partial_{z})
			=g^{L}_{f}(\partial_{y},\partial_{y})=1~\text{and}~g^{L}_{f}(\partial_{z},\partial_{z})=f.
	\end{align*}
	Then we have $\nabla^{g^{L}_{f}}_{\cdot}{(\partial_{x}+a\partial_{y})}=\alpha(\cdot)(\partial_{x}+a\partial_{y})$ such that
	$\alpha|_{\Xi^{\perp}}=0$ and $Hol(X,g^{L}_{f})=\R \ltimes \R$ for a suitable choice of $f$.
	In particular, $\mathcal{L}^{\perp}=T^{2}$ for any leaf of $\mathcal{X}^{\perp}$ and all leaves of $\mathcal{X}$ are dense in $T^{2}$.
	Finally, $(X,g^{L}_{f},\partial_{x}+a\partial_{y},\mbox{span}\{\partial_{y}\})$ is horizontal.
\end{example}
\begin{proof}
	Define $V:=\partial_{x}+a\partial_{y}$. Then $[V,\partial_{y}]=[V,\partial_{z}]=[\partial_{y},\partial_{z}]=0$, i.e., locally we
	have coordinates $(\tilde{x},\tilde{y},\tilde{z})$ such that $V=\partial_{\tilde{x}}$, $\partial_{y}=\partial_{\tilde{y}}$ and
	$\partial_{z}=\partial_{\tilde{z}}$. In these local coordinates the Lorentzian metric is given by
	\begin{equation*}
		g^{L}_{f} = 2d\tilde{x}d\tilde{z} + fd\tilde{z}^{2} + d\tilde{y}^{2}.
	\end{equation*}
	Since $V=\partial_{\tilde{x}}$ we conclude that $\nabla^{L}_{\cdot}{V}=\alpha(\cdot)V$ and $\alpha|_{\Xi^{\perp}}=0$. If the restriction of
	$g^{L}$ to the coordinate neighborhood $U$ is weakly irreducible so is $(X,g^{L}_{f})$. Using a suitable choice of $f \in C^{\infty}(T^{2})$ we
	derive a weakly irreducible neighborhood $(U,g^{L}_{f})$ and $\frac{\partial^{2}f}{\partial \tilde{x}^{2}} \neq 0$.\par
	The maximal integral curves of $V$ are dense in $T^{2}$ since $a \in \R \setminus \Q$. Moreover,
	$\Xi^{\perp}=\mbox{span}\{V,\partial_{y}\}$, i.e., $\mathcal{L}^{\perp}=T^{2}$. The global vector field $\partial_{y}$ defines
	a non-canonical realization $S$ of the screen bundle. Hence, $S$ admits a global nowhere vanishing section which is $\nabla^{S}$-parallel and
	we conclude $Hol(X,g^{L}_{f})=\R \ltimes \R$. Finally, $\partial_{y}=\partial_{\tilde{y}}$ and the local coordinate structure imply
	$\nabla^{L}_{V}{\partial_{y}}=0$, i.e., $(X,g^{L}_{f})$ is horizontal.
\end{proof}
Example \ref{bazaikin-example} is in fact horizontal if $V:=\partial_{x}$ and $S:=TM$ and another class of globally hyperbolic decent spacetimes was
constructed in \cite{MR2350042}. Using the notation of \cite{MR2350042} we derive horizontal spacetimes if $S:=TF$.
%
Finally, Tom Krantz constructed another class of weakly irreducible spacetimes which are almost horizontal by \cite[Prop. 4]{MR2578019}.
\section{Ricci Comparison for Decent Spacetimes}
Let $(X,\mathcal{F})$ be a foliated manifold. A differential $r$-form $\omega$ on $X$ is $X$-basic or basic if $V \lrcorner \omega =0$ and
$L_{V}\omega =0$. We derive a sheaf of germs of basic $r$-forms and write $\Lambda^{r}_{B}\mathcal{F}$ for its space of global sections.
By definition, if $\omega$ is basic so is $d\omega$. Hence, we have the {\em basic cohomology ring} $H^{*}_{B}(X,\mathcal{F})$ of $(X,\mathcal{F})$.
If $X$ is connected we have $H^{0}_{B}(X,\mathcal{F})=\R$ and a group monomorphism $H^{1}_{B}(X,\mathcal{F}) \hookrightarrow H^{1}(X,\R)$ induced by
$\Lambda^{1}_{B}\mathcal{F} \hookrightarrow \Lambda^{1}TX$ (cf. \cite{MR2382957}[Prop. 2.4.1]).\par
Let $(X,g^{L},V)$ be a decent spacetime and consider the $(V,S)$-metric $g^{R}$ for some realization $S$ of the screen bundle.
If $X$ is compact we have $b_{1}(X) \geq 1$ by Cor.\ref{all-leaves-closed} and if $\mathcal{X}^{\perp}$ admits a compact leaf the projection onto
the space of leaves is a fiber bundle $X \rightarrow S^{1}$ whose fibers are the leaves of $\mathcal{X}^{\perp}$. Hence, $X$ is a mapping torus, i.e.,
if $\mathcal{L}^{\perp}$ is a leaf of $\mathcal{X}^{\perp}$ there is a diffeomorphism $F$ of $\mathcal{L}^{\perp}$ such that
$X = \Mod{\mathcal{L}^{\perp} \times [0,1]}{\sim}$ where $(p,0)\sim (F(p),1)$. Using Cor. \ref{all-leaves-closed} we have
$b_{1}(X)=b_{1}(\mathcal{L}^{\perp})+1$. For the higher Betti numbers a Mayer-Vietoris argument yields the following exact sequence in
singular homology
\begin{equation*}
	\longrightarrow H_{i}(\mathcal{L}^{\perp}) \stackrel{Id -F^{i}_{*}}{\longrightarrow} H_{i}(\mathcal{L}^{\perp})
				\stackrel{\iota_{*}}{\longrightarrow} H_{i}(X) \longrightarrow H_{i-1}(\mathcal{L}^{\perp})
			\stackrel{Id -F^{i-1}_{*}}{\longrightarrow} H_{i-1}(\mathcal{L}^{\perp}) \longrightarrow
\end{equation*}
where $F^{i}_{*}$ is the morphism induced by $F$ and $\iota: \mathcal{L}^{\perp} \hookrightarrow X$ is the inclusion.
On the other hand, if $X$ is non-compact and all leaves of $\mathcal{X}^{\perp}$ are compact then the natural projection induces a fiber bundle map
$X \rightarrow \R$ \cite{MR1453120}, i.e., $X \cong \mathcal{L}^{\perp} \times \R$ and $b_{i}(X)=b_{i}(\mathcal{L}^{\perp})$.\par
Consider an arbitrary almost decent spacetime $(X,g^{L},V)$ and suppose $\mathcal{X}^{\perp}$ admits a compact leaf $\mathcal{L}^{\perp}$. If
$\tilde{g}^{R}$ is a bundle-like Riemannian metric on the compact foliated manifold $(\mathcal{L}^{\perp},\mathcal{X}|_{\mathcal{L}^{\perp}})$ and
$E^{\perp}$ is the $\tilde{g}^{R}$-orthogonal complement of $V$ we define the mean curvature 1-form by
$\kappa_{\tilde{g}^{R}}:=\tilde{g}^{R}(pr_{E^{\perp}}(\nabla^{\tilde{g}^{R}}_{\frac{V}{\|V\|_{\tilde{g}^{R}}}}{\frac{V}{\|V\|_{\tilde{g}^{R}}}}),\cdot)$.
Since $\mathcal{L}^{\perp}$ is compact \cite{MR1657170} and \cite{MR1811936} imply the existence of a bundle-like Riemannian metric $\tilde{g}^{B}$ on
$\mathcal{L}^{\perp}$ such that $\kappa_{\tilde{g}^{B}}$ is basic and harmonic w.r.t. the basic Laplacian.
In this case, the Euler form $\mathbf{e}$ of $(\mathcal{L}^{\perp},\mathcal{X}|_{\mathcal{L}^{\perp}},\tilde{g}^{B})$ is defined using Rummler's formula
\begin{equation*}
	d(\tilde{g}^{B}(\frac{V}{\|V\|},\cdot)) = \tilde{g}^{B}(\frac{V}{\|V\|},\cdot) \wedge \kappa_{\tilde{g}^{B}} + \mathbf{e}.
\end{equation*}
In \cite{MR1871045} Royo Prieto proved the existence of a Gysin sequence for $(\mathcal{L}^{\perp},\mathcal{X}|_{\mathcal{L}^{\perp}},g^{B})$ relating
the basic cohomology of $(\mathcal{L}^{\perp},\mathcal{X}|_{\mathcal{L}^{\perp}})$ to the cohomology of $\mathcal{L}^{\perp}$
by\footnote{We have seen in Lemma \ref{horizontal-isometric-flow} that
			$(\mathcal{L}^{\perp},\mathcal{X}|_{\mathcal{L}^{\perp}},g^{R}|_{\mathcal{L}^{\perp}})$ is an isometric Riemannian flow
			if $g^{R}$ is the $(V,S)$-metric of an almost horizontal spacetime $(X,g^{L},V,S)$. Thus,
			$\kappa_{g^{R}|_{\mathcal{L}^{\perp}}}=0$ and the Gysin sequence for
			$(\mathcal{L}^{\perp},\mathcal{X}|_{\mathcal{L}^{\perp}},g^{R}|_{\mathcal{L}^{\perp}})$ is given by
			\begin{equation*}
				\cdots \rightarrow H^{i}_{B}(\mathcal{X}|_{\mathcal{L}^{\perp}}) \rightarrow H^{i}(\mathcal{L}^{\perp},\R) \rightarrow
					H^{i-1}_{B}(\mathcal{X}|_{\mathcal{L}^{\perp}}) \stackrel{\delta}{\rightarrow}
						H^{i+1}_{B}(\mathcal{X}|_{\mathcal{L}^{\perp}}) \rightarrow \cdots
			\end{equation*}
			where $\delta = [dg^{L}(Z,\cdot) \wedge \cdot]$ (cf. \cite[Thm. 7.2.1]{MR2382957}). In particular, the Euler class is given by
			$[dg^{L}(Z,\cdot)] \in H^{2}_{B}(\mathcal{X}|_{\mathcal{L}^{\perp}})$.}
\begin{equation*}
	\cdots \rightarrow H^{i}_{B}(\mathcal{X}|_{\mathcal{L}^{\perp}}) \rightarrow H^{i}(\mathcal{L}^{\perp},\R) \rightarrow
			H^{i-1}_{d-\kappa_{\tilde{g}^{B}}}(\mathcal{X}|_{\mathcal{L}^{\perp}}) \stackrel{[\cdot \wedge \mathbf{e}]}{\rightarrow}
				H^{i+1}_{B}(\mathcal{X}|_{\mathcal{L}^{\perp}}) \rightarrow \cdots
\end{equation*}
Here, we write $H^{*}_{d-\kappa_{\tilde{g}^{B}}}(\mathcal{X}|_{\mathcal{L}^{\perp}})$ for the {\em dual basic cohomology} which can be defined in
the following way. If $(X,\mathcal{F},\tilde{g}^{B})$ is a Riemannian flow whose mean curvature $\kappa_{\tilde{g}^{B}}$ 1-form is basic and
harmonic w.r.t. the basic Laplacian then $H^{*}_{d-\kappa_{\tilde{g}^{B}}}(X,\mathcal{F})$ is the cohomology of the complex
$(\Lambda_{B}^{*}\mathcal{F},d-\kappa_{\tilde{g}^{B}}\wedge \cdot)$. It can be shown that
$H^{i}_{B}(X,\mathcal{F}) \cong H^{\dim\mathcal{L}^{\perp} -i}_{d-\kappa_{\tilde{g}^{B}}}(X,\mathcal{F})$ for all
$i \geq 0$ \cite[Sec. 1.5]{habib-richardson-2010}.\par
For a Riemannian flow $(X,\mathcal{F},\tilde{g}^{B})$ whose mean curvature 1-form $\kappa_{\tilde{g}^{B}}$ is basic and harmonic consider the twisted
differential $d_{\kappa}:=d-\frac{1}{2}\kappa_{\tilde{g}^{B}}\wedge$. The {\em twisted basic cohomology} $H^{*}_{tw}(X,\mathcal{F})$ is defined as the
cohomology of the complex $(\Lambda_{B}^{*}\mathcal{F},d_{\kappa})$ and if $\delta_{\kappa}$ denotes the formal $L^{2}$-adjoint of $d_{\kappa}$ on
$\Lambda_{B}^{*}\mathcal{F}$ the twisted basic Laplacian is defined by $\Delta_{\kappa}:=d_{\kappa}\delta_{\kappa}+\delta_{\kappa}d_{\kappa}$. In
\cite{habib-richardson-2010} Habib and Richardson proved a Hodge decomposition theorem for $\Delta_{\kappa}$ and the following Weitzenb\"{o}ck formula
for any basic form $\varphi \in \Lambda_{B}^{*}\mathcal{F}$
\begin{equation*}
	\Delta_{\kappa} \varphi = {\nabla^{T}}^{*}\nabla^{T} \varphi + \sum_{i,j}{e^{j} \wedge e_{i} \lrcorner R^{T}(e_{i},e_{j})\varphi}
					+\frac{1}{4}|\kappa_{\tilde{g}^{B}}|^{2}\varphi.
\end{equation*}
Here we write ${\nabla^{T}}^{*}$ for the formal $L^{2}$-adjoint of the transverse Levi-Civita connection on basic forms and
$R^{T}(e_{i},e_{j}):=[\nabla^{T}_{e_{i}},\nabla^{T}_{e_{j}}]-\nabla^{T}_{[e_{i},e_{j}]}$ where $(e_{1},\ldots,e_{\dim X -1})$ is a transverse orthonormal
frame. For a basic 1-form $\varphi$ Habib and Richardson proved
\begin{equation*}
	\skal{\Delta_{\kappa}\varphi}{\varphi} = \skal{{\nabla^{T}}^{*}\nabla^{T}\varphi}{\varphi} + Ric^{T}(\varphi^{\sharp},\varphi^{\sharp})
							+\frac{1}{4}|\kappa_{\tilde{g}^{B}}|^{2}|\varphi|^{2},
\end{equation*}
where $Ric^{T}$ is the Ricci curvature of $\nabla^{T}$.\par
If $(X,g^{L},V)$ is an almost decent spacetime and $S$ a realization of the screen bundle then Cor. \ref{transverse-levi-civita} implies
$\nabla^{S}|_{\mathcal{L}^{\perp}}=\nabla^{T}$ where $\nabla^{T}$ is transverse Levi-Civita connection of the Riemannian flow
$(\mathcal{L}^{\perp},\mathcal{X}|_{\mathcal{L}^{\perp}},g^{R}|_{\mathcal{L}^{\perp}})$. If $(Y_{1},\ldots,Y_{\dim S})$ is a local orthonormal frame
of $S$ we write $E_{\pm}:=\frac{1}{\sqrt{2}}(V \pm Z)$ and conclude
\begin{align*}
	Ric^{L}(Y_{\alpha},Y_{\beta}) &= -g^{L}(R^{L}(Y_{\alpha},E_{-})E_{-},Y_{\beta}))
						+\sum_{k=1}^{\dim S}{g^{L}(R^{L}(Y_{\alpha},Y_{k})Y_{k},Y_{\beta})}\\
					&\qquad +g^{L}(R^{L}(Y_{\alpha},E_{+})E_{+},Y_{\beta}))\\
			&= \underbrace{g^{L}(R^{L}(Y_{\alpha},V)Z,Y_{\beta})}_{=-g^{L}(R^{L}(Z,Y_{\beta})V,Y_{\alpha})}
					+ g^{L}(\underbrace{R^{L}(Y_{\alpha},Z)V}_{\in \Xi},Y_{\beta})\\
					&\qquad + \sum_{k=1}^{\dim S}{g^{L}(R^{L}(Y_{\alpha},Y_{k})Y_{k},Y_{\beta})}\\
			&= \sum_{k=1}^{\dim S}{g^{R}(R^{S}(Y_{\alpha},Y_{k})Y_{k},Y_{\beta})} = Ric^{T}(Y_{\alpha},Y_{\beta})
\end{align*}
for the Ricci curvature $Ric^{L}$ of $(X,g^{L})$. Note that $Ric^{L}(V,\cdot)|_{\Xi^{\perp}}=0$.
\begin{proposition}\label{cohom-of-a-leaf}
	Let $(X,g^{L},V)$ be an almost decent spacetime and $\mathcal{L}^{\perp}$ a compact leaf of $\mathcal{X}^{\perp}$.
	If $Ric^{L}(W,W) \geq 0$ for all $W \in T\mathcal{L}^{\perp}$ then $b_{1}(\mathcal{L}^{\perp}) \leq \dim\mathcal{L}^{\perp}$.
	If additionally $Ric^{L}_{q}(W,W) > 0$ for some $q \in \mathcal{L}^{\perp}$ and all $W \in S_{q}$ then
	$b_{1}(\mathcal{L}^{\perp}) \leq 1$.
\end{proposition}
\begin{proof}
	Suppose $Ric^{L}(W,W) \geq 0$ and let $\tilde{g}^{B}$ be a bundle-like Riemannian metric
	on $(\mathcal{L}^{\perp},\mathcal{X}|_{\mathcal{L}^{\perp}})$ having a basic and harmonic mean curvature form $\kappa_{\tilde{g}^{B}}$.
	By the Hodge theorem for the basic Laplacian \cite{MR1146730} a class $[\varphi] \in H^{1}_{B}(\mathcal{X}|_{\mathcal{L}^{\perp}})$ can
	be represented by a basic 1-form $\varphi$ such that $d\varphi=\delta_{B}\varphi=0$ where $\delta_{B}$ is the $L^{2}$-adjoint of
	$d|_{\Lambda^{1}_{B}\mathcal{X}|_{\mathcal{L}^{\perp}}}$. In this case, the Weitzenb\"{o}ck formula (cf. \cite[Thm. 6.16]{habib-richardson-2010})
	has the form $0 = \int_{\mathcal{L}^{\perp}}{|\nabla^{T}\varphi|^{2}} + \int_{\mathcal{L}^{\perp}}{Ric^{T}(\varphi^{\sharp},\varphi^{\sharp})}$.
	Hence, $\nabla^{T}\varphi =0$, i.e., $\dim H^{1}_{B}(\mathcal{X}|_{\mathcal{L}^{\perp}}) \leq \dim \mathcal{S}$ for dimensional reasons.
	If $Ric^{L}_{q}(Y,Y) > 0$ at $q \in \mathcal{L}^{\perp}$ \cite[Cor. 6.17]{habib-richardson-2010} implies
	$H^{1}_{B}(\mathcal{X}|_{\mathcal{L}^{\perp}})=0$. Since
	$H^{0}_{d-\kappa_{\tilde{g}^{B}}}(\mathcal{X}|_{\mathcal{L}^{\perp}}) = H^{\dim \mathcal{L}^{\perp}-1}_{B}(\mathcal{X}|_{\mathcal{L}^{\perp}})
										\in \{\R,0\}$ and
	\begin{equation*}
			0 \rightarrow H^{1}_{B}(\mathcal{X}|_{\mathcal{L}^{\perp}}) \rightarrow H^{1}(\mathcal{L}^{\perp},\R) \rightarrow
			H^{0}_{d-\kappa_{\tilde{g}^{B}}}(\mathcal{X}|_{\mathcal{L}^{\perp}}) \stackrel{[\cdot \wedge \mathbf{e}]}{\rightarrow}
				H^{2}_{B}(\mathcal{X}|_{\mathcal{L}^{\perp}})
	\end{equation*}
	we conclude $b_{1}(\mathcal{L}^{\perp}) \leq \dim H^{1}_{B}(\mathcal{X}|_{\mathcal{L}^{\perp}}) +1$.
\end{proof}
\begin{corollary}\label{bounds-for-cohom}
	Let $(X,g^{L},V)$ be a decent spacetime and $\mathcal{L}^{\perp}$ a leaf of $\mathcal{X}^{\perp}$. Suppose $Ric^{L}(W,W) \geq 0$ for
	all $W \in T\mathcal{L}^{\perp}$.
	\begin{enumerate}
		\item
		If $X$ is compact and $\mathcal{X}^{\perp}$ admits a compact leaf then $1 \leq b_{1}(X) \leq \dim X$.
		\item
		If $X$ is non-compact and all leaves of $\mathcal{X}^{\perp}$ are compact then $0 \leq b_{1}(X) \leq \dim X -1$.
	\end{enumerate}
	Moreover, if $Ric^{L}_{q}(W,W) > 0$ for some $q \in \mathcal{L}^{\perp}$ and all $W \in S_{q}$ the bounds are
	$1 \leq b_{1}(X) \leq 2$ and $0 \leq b_{1}(X) \leq 1$ respectively.
\end{corollary}
\begin{proof}
	Using the Mayer-Vietoris argument and Prop. \ref{cohom-of-a-leaf} we conclude $b_{1}(X) \leq b_{1}(\mathcal{L}^{\perp}) + 1 \leq \dim X$ if $X$
	is compact. In the non-compact case we observed $X \cong \mathcal{L}^{\perp} \times \R$, i.e.,
	$b_{1}(X) = b_{1}(\mathcal{L}^{\perp}) \leq \dim X -1$.
\end{proof}
\begin{proposition}
	The bounds in Cor. \ref{bounds-for-cohom} are optimal.
\end{proposition}
\begin{proof}
	First, we consider the upper bounds. If $(M,g)$ is a compact Riemannian manifold we derive weakly irreducible Lorentzian metrics on
	$S^{1} \times M \times \R$ and on $S^{1} \times M \times S^{1}$ as follows: If $\partial_{x}$ is the coordinate field on $S^{1}$ define
	$g^{L}:= 2dxdz +fdz^{2} + g$ where $\partial_{z}$ is the coordinate field of the last factor and $f \in C^{\infty}(M)$ is suitable.
	Then $\Xi = TS^{1}$, $\mathcal{L}^{\perp} \cong S^{1} \times M$ and $\nabla^{S}|_{\mathcal{L}^{\perp}}$ is flat if $(M,g)$ is the flat torus.\par
	For the second statement let $(M,g)$ be a compact simply connected Riemannian manifold with strictly positive Ricci curvature. Hence,
	$Ric^{T} > 0$ and the upper bounds are optimal.\par
	For the lower bounds let $(M,g)$ be a compact simply connected Calabi-Yau manifold, i.e., $Hol(M,g)=SU(n)$. Consider the total space
	$\tilde{M}$ of the $S^{1}$-bundle given by $0 \neq \alpha \in H^{1,1}_{prim}(M)\cap H^{2}(M,\Z)$. It is shown in \cite[Cor. 4.4]{laerz-2008-a}
	that $X:= \tilde{M} \times S^{1}$ and $X:= \tilde{M} \times S^{1}$ admit weakly irreducible Lorentzian metrics $g^{L}$ such that
	$Hol(X,g^{L})= SU(n) \ltimes \R^{2n}$. In particular, $\mathcal{L}^{\perp} \cong \tilde{M}$ and $\nabla^{S}|_{\mathcal{L}^{\perp}}$ is Ricci
	flat. The Gysin sequence for the $S^{1}$-bundles implies $b_{1}(\mathcal{L}^{\perp})=0$ since $0 \neq \alpha \in H^{2}(M,\R)$.\par
	Finally, we study the lower bounds if $Ric^{T}>0$. Let $(M,g)$ be a compact simply connected Riemannian manifold with strictly positive Ricci
	curvature and let $\alpha \in H^{2}(M,\Z)$ be a generator. Using the construction in \cite{laerz-2008-a} we derive weakly irreducible Lorentzian
	metrics on $X= \tilde{M} \times S^{1}$ and on $X=\tilde{M} \times \R$ where $\tilde{M}$ is the total space of the $S^{1}$-bundle given by
	$\alpha$. Moreover, $Ric^{T}|_{S \times S} = Ric(M,g)$ and $\mathcal{L}^{\perp} \cong \tilde{M}$. Hence, $b_{1}(\mathcal{L}^{\perp})=0$ by the
	Gysin sequence.
\end{proof}
We say $(X,g^{L})$ satisfies the {\em strong energy (timelike convergence) condition} at $p \in X$ if $Ric^{L}_{p}(W,W) \geq 0$ for any timelike vector
$W \in T_{p}X$. If $\nabla^{L}_{\cdot}{V}=0$ we have $Ric^{L}(V,\cdot)=0$ and $Ric^{L}(Z,Z) = \sum_{k}{g^{L}(R^{L}(Z,Y_{k})Y_{k},Z)}$ as well as
$Ric^{L}(Z,Y_{i}) = \sum_{k}{g^{L}(R^{S}(Z,Y_{k})Y_{k},Y_{i})}$.
\begin{remark}
	Let $(X,g^{L},V)$ be a decent spacetime such that $\nabla^{L}_{\cdot}{V}=0$ and $p \in X$. If $Ric^{L}_{p}(Z,Z)=0$ and
	$\sum_{k}{R^{S}_{p}(Z,Y_{k})Y_{k}}=0$ then $(X,g^{L})$ satisfies the strong energy condition at $p \in X$ if and only if
	$Ric^{L}_{p}(W,W) \geq 0$ for all $W \in \Xi^{\perp}_{p}$.\qed
\end{remark}
\section{Screen Holonomy and the Topology of Decent Spacetimes}
If there is an integrable realization of the screen bundle the Blumenthal-Hebda decomposition theorem \cite{MR699494} immediately implies
\begin{corollary}
	Let $(X,g^{L},V)$ be an almost decent spacetime and $\mathcal{L}^{\perp}$ a leaf of $\mathcal{X}^{\perp}$. Suppose $S$ is an integrable
	realization of the screen bundle and $p \in \mathcal{L}^{\perp}$.
	\begin{enumerate}
		\item
		If $g^{R}|_{\mathcal{L}^{\perp}}$ is complete then $\tilde{\mathcal{L}}^{\perp}= \R \times \tilde{\mathcal{S}}$ where
		$\tilde{\mathcal{L}}^{\perp}$ is the universal cover of $\mathcal{L}^{\perp}$ and $\tilde{\mathcal{S}}$ is the universal cover of a
		leaf of $S|_{\mathcal{L}^{\perp}}$.
		\item
		If $(X,g^{L},V,S)$ is horizontal such that $(L_{Z}g^{L})|_{S \times S}=0$ and $g^{R}$ is complete then
		$\tilde{X} = \R^{2} \times \tilde{\mathcal{S}}$ where $\tilde{X}$ is the universal cover of $X$.
		\item
		In both cases, if $Hol^{0}(\nabla^{S}) = H_{1} \times H_{2}$ then
		$\tilde{\mathcal{S}}=\tilde{\mathcal{S}}_{1} \times \tilde{\mathcal{S}}_{2}$ as Riemannian manifolds and
		$Hol(\tilde{\mathcal{S}}_{i}) \subset H_{i}$.
	\end{enumerate}
\end{corollary}
\begin{proof}
	Since $g^{R}|_{\mathcal{L}^{\perp}}$ is bundle-like for $(\mathcal{L}^{\perp},\mathcal{X}|_{\mathcal{L}^{\perp}})$ the leaves of
	$S|_{\mathcal{L}^{\perp}}$ are totally geodesic in $\mathcal{L}^{\perp}$ and we can apply the Blumenthal-Hebda theorem.\par
	As we have seen above $V$ and $Z$ induce a 2-dimensional foliation on $X$ if $(X,g^{L},V,S)$ is horizontal and $g^{R}$ is bundle-like for this
	foliation if $(L_{Z}g^{L})|_{S \times S}=0$. The Blumenthal-Hebda theorem implies $\tilde{X}=M \times \tilde{\mathcal{S}}$ where
	$M$ the universal cover of a leaf of the foliation induced by $V$ and $Z$. Since $M$ is a simply connected parallelizable surface the
	uniformization theorem implies $M \cong \R^{2}$.\par
	The last statement follows from the de Rham decomposition theorem since
	$Hol(\nabla^{S}|_{\tilde{\mathcal{S}}}) \subset Hol(\nabla^{S}|_{\tilde{\mathcal{L}}^{\perp}}) \subset Hol^{0}(\nabla^{S})$.
\end{proof}
\begin{remark}
	If $M$ is a compact simply connected manifold and $X \rightarrow M$ is an $S^{1}$-bundle whose Euler class is a generator of $H^{2}(M,\Z)$ then
	the universal cover of $X$ is compact.
	Using \cite{laerz-2008-a} we derive a decent Lorentzian metric on $X \times \R$ which does not admit an integrable realization of the screen
	bundle. In fact, using the Milnor-Wood inequality \cite{MR0293655} we can construct 4-dimensional decent spacetimes such that
	$(\mathcal{L}^{\perp},\mathcal{X}|_{\mathcal{L}^{\perp}})$ does not admit a transverse foliation.\qed
\end{remark}
By Cor. \ref{transverse-levi-civita} and \cite[Prop. 1.6]{MR0370617} the foliated manifold $(\mathcal{L}^{\perp},\mathcal{X}|_{\mathcal{L}^{\perp}})$
admits a transverse $G$-structure if $Hol(\nabla^{S}|_{\mathcal{L}^{\perp}}) \subset G$. Note that
$Hol(\nabla^{S}|_{\mathcal{L}^{\perp}}) \subset Hol(\nabla^{S})$. The classification of Lorentzian holonomy representations, i.e., representations of
$\liealg{hol}(X,g^{L})$ has been achieved by Leistner in \cite{MR2331527} and the hard part is to show that $\liealg{hol}(\nabla^{S})$ acts as a
Riemannian holonomy representation. Galaev \cite{MR2264404} constructed real analytic decent spacetimes for all possible representations of
$\liealg{hol}(\nabla^{S})$ for which $\liealg{hol}(\nabla^{S}|_{\mathcal{L}^{\perp}})$ is trivial for any leaf $\mathcal{L}^{\perp}$ of
$\mathcal{X}^{\perp}$. Since $\liealg{hol}(\nabla^{S})$ has the Borel-Lichn\'erowicz property (cf. \cite[Thm. 2.1]{MR2331527}) we have decompositions
\begin{equation*}
	S_{p}= E_{0} \oplus \ldots \oplus E_{\ell} \quad \text{and}
						\quad Hol^{0}_{p}(\nabla^{S}) = H_{1} \oplus \ldots \oplus H_{\ell}
\end{equation*}
where $H_{j}$ acts irreducibly on $E_{j}$ for $j \geq 1$. If $\gamma:[0,1] \rightarrow X$ is a piecewise smooth curve such that $\gamma(0)=p$ and if
$\tau_{\gamma}^{S}$ is the parallel displacement w.r.t. $\nabla^{S}$ along $\gamma$ we define
$R_{p}^{\tau_{\gamma}^{S}}(v,w):= {\tau_{\gamma}^{S}}^{-1} \circ R^{S}_{\gamma(1)}(w,v) \circ \tau_{\gamma}^{S}$ for $v,w \in S_{\gamma(1)}$. The
Ambrose-Singer theorem and $R^{S}(V,\Xi^{\perp})=0$ imply
\begin{equation*}
	\liealg{hol}_{p}(\nabla^{S}|_{\mathcal{L}^{\perp}}) =\mbox{span}\{ R_{p}^{\tau_{\gamma}^{S}}(\tau_{\gamma}^{S}v,\tau_{\gamma}^{S}w): v,w\in S_{p},
									\gamma:[0,1] \rightarrow \mathcal{L}^{\perp} \}.
\end{equation*}
Moreover, each $R_{p}^{\tau_{\gamma}^{S}}(\tau_{\gamma}^{S}(\cdot),\tau_{\gamma}^{S}(\cdot))$ is an algebraic curvature tensor on $S_{p}$. Hence,
$\liealg{hol}_{p}(\nabla^{S}|_{\mathcal{L}^{\perp}})$ is a Berger algebra in $\liealg{so}(S_{p})$, i.e., it acts as a Riemannian holonomy
representation. Since each subspace $E_{j}$ is $Hol^{0}_{p}(\nabla^{S}|_{\mathcal{L}^{\perp}})$-invariant we may consider
\begin{equation*}
	\mathcal{K}(E_{j}):=\mbox{span}\{ R_{p}^{\tau_{\gamma}^{S}}(\tau_{\gamma}^{S}(\cdot),\tau_{\gamma}^{S}(\cdot))|_{E_{j}\times E_{j}\times E_{j}}\}.
\end{equation*}
Suppose $0 \neq \tilde{R} \in \mathcal{K}(E_{k})$. Then $(E_{k},\tilde{R},H_{k})$ is an irreducible holonomy system and Simons' theorem \cite{MR0148010}
implies that $H_{k}$ acts on $E_{k}$ as a Riemannian holonomy representation.
\begin{lemma}\label{flat-part-split-off}
	Let $(X,g^{L},V)$ be an almost decent spacetime and $S$ a realization of the screen bundle. Suppose there is a leaf
	$\mathcal{L}^{\perp}$ of $\mathcal{X}^{\perp}$ such that $(\mathcal{L}^{\perp},g^{R}|_{\mathcal{L}^{\perp}})$ is complete.
	If $p \in \mathcal{L}^{\perp}$ and $\mathcal{K}(E_{k})=0$ then $\tilde{\mathcal{L}}^{\perp} = A \times \R^{\dim E_{k}}$ where
	$\tilde{\mathcal{L}^{\perp}}$ is the universal cover of $\mathcal{L}^{\perp}$.
\end{lemma}
\begin{proof}
	Consider the foliated manifold $(\tilde{\mathcal{L}}^{\perp},\tilde{\mathcal{X}}|_{\mathcal{L}^{\perp}},\tilde{g}^{R}|_{\mathcal{L}^{\perp}})$
	and the lifted connection $\tilde{\nabla}^{S}|_{\mathcal{L}^{\perp}}$. Since $\tilde{\mathcal{L}}^{\perp}$ is simply connected we have
	$\tilde{\nabla}^{S}|_{\mathcal{L}^{\perp}}$-parallel orthonormal sections
	$Y_{1},\ldots,Y_{\dim E_{k}} \in \Gamma(\tilde{\mathcal{L}}^{\perp},\tilde{S})$. An integral curve of any $Y_{i}$ is a horizontal
	$\tilde{g}^{R}|_{\mathcal{L}^{\perp}}$-geodesic. Hence, each $Y_{i}$ is a complete vector field on $\tilde{\mathcal{L}}^{\perp}$.
	Define $\mathcal{T}^{Y_{1}} := \mbox{span}\{Y_{1}\}^{\perp} \subset T\tilde{\mathcal{L}}^{\perp}$. If $W \in \Gamma(U,\mathcal{T}^{Y_{1}})$ is
	a local section then $[W,Y_{1}] \in \tilde{\nabla}^{S}_{W}{Y_{1}} - \tilde{\nabla}^{S}_{Y_{1}}{W} + \tilde{\mathcal{X}}|_{\mathcal{L}^{\perp}}
					\subset \mathcal{T}^{Y_{1}}$.
	Moreover, $\tilde{\nabla}^{S}_{\cdot}\mathcal{T}^{Y_{1}} \subset \mathcal{T}^{Y_{1}}$. Thus, $\mathcal{T}^{Y_{1}}$ induces a transversely
	parallelizable codimension one foliation in $\mathcal{L}^{\perp}$ and \cite[Prop. 5.3]{MR0370617} implies
	$\mathcal{L}^{\perp} = A_{Y_{1}} \times \R$ where $A_{Y_{1}}$ is a leaf of $\mathcal{T}^{Y_{1}}$. For $i \geq 2$ we restrict the vector fields
	$Y_{i}$ to $A_{Y_{1}}$. As above, we derive a transversely parallelizable codimension one foliation on $A_{Y_{1}}$ induced by
	$\mathcal{T}^{Y_{2}}:=\mbox{span}\{Y_{2}|_{A_{Y_{1}}}\}^{\perp}$ and $Y_{2}$ is a complete transverse vector field. Inductively, we have
	$\tilde{\mathcal{L}}^{\perp} = A \times \R^{\dim E_{k}}$
\end{proof}
\begin{theorem}
	Let $(X,g^{L})$ be a time-orientable Lorentzian manifold such that $\liealg{hol}(X,g^{L})$ acts weakly irreducible and reducible. Suppose
	the associated foliation $\mathcal{X}^{\perp}$ admits a compact leaf $\mathcal{L}^{\perp}$ such that $\pi_{1}(\mathcal{L}^{\perp})$ is finite.
	Then $\liealg{hol}(X,g^{L})$ belongs to one of the following types where $\liealg{g}:=\liealg{hol}(\nabla^{S})$.
	\begin{itemize}
		\item \noindent
		Type 1: $\liealg{hol}(X,g^{L}) = (\R \oplus \liealg{g}) \ltimes \R^{\dim X-2}$
		\item \noindent
		Type 2: $\liealg{hol}(X,g^{L}) = \liealg{g} \ltimes \R^{\dim X -2}$
		\item \noindent
		Type 3:
		\begin{equation*}
			\liealg{hol}(X,g^{L}) = \left\lbrace \begin{pmatrix}
							\varphi(A) & w^T & 0\\
							0 & A & -w\\
							0 & 0 & -\varphi(A)
						\end{pmatrix}: A \in \liealg{g},~w \in \R^{q} \right\rbrace
		\end{equation*}
		where $\varphi: \liealg{g} \twoheadrightarrow \R$ is an epimorphism satisfying
		$\varphi|_{[\liealg{g},\liealg{g}]}=0$.
	\end{itemize}
	Moreover, identifying $\liealg{g} \subset \liealg{so}(\dim X -2)$ there are decompositions
	\begin{equation*}
		\R^{\dim X -2} = F_{1} \oplus \ldots \oplus F_{\ell} \quad \text{and}
							\quad \liealg{g} = \liealg{g}_{1} \oplus \ldots \oplus \liealg{g}_{\ell}
	\end{equation*}
	such that each $\liealg{g}_{j}$ acts trivially on $F_{i}$ for $i \neq j$ and as an irreducible Riemannian holonomy representation on $F_{j}$.
	In particular, $\liealg{g}$ does not act trivially on any subspace of $\R^{\dim X -2}$.
\end{theorem}
\begin{proof}
	The universal cover of $\mathcal{L}^{\perp}$ is compact and Lemma \ref{flat-part-split-off} implies that $\liealg{g}$ does not act trivially on
	any subspace of $\R^{\dim X -2}$. It is shown in \cite{MR1216527} that if $\liealg{hol}(X,g^{L})$ does not belong to one of the three types then
	it is given as follows. There is $0< \ell < q$ such that $\R^{q} = \R^{\ell} \oplus \R^{q-\ell}$, $\liealg{g} \subset \liealg{so}(\ell)$ and
	\begin{equation*}
		\liealg{hol}(X,g^{L}) = \left\lbrace \begin{pmatrix}
							0 & \psi(A)^T & w^T & 0\\
							0 & 0 & 0 & -\psi(A)\\
							0 & 0 & A & -w\\
							0 & 0 & 0 & 0
						\end{pmatrix}: A \in \liealg{g},~w \in \R^{\ell}  \right\rbrace
	\end{equation*}
	for some epimorphism $\psi: \liealg{g} \twoheadrightarrow \R^{q-\ell}$ satisfying $\psi|_{[\liealg{g},\liealg{g}]}=0$. Since $\liealg{g}$ acts
	trivially on $\R^{q-\ell}$ we derive a contradiction.
\end{proof}
Let $A$ be a global section of some tensor bundle of $S$ and suppose that $A|_{\mathcal{L}^{\perp}}$ is invariant under the action of
$Hol(\nabla^{S}|_{\mathcal{L}^{\perp}})$ for any leaf $\mathcal{L}^{\perp}$ of $\mathcal{X}^{\perp}$. Then $A$ is invariant under the action of
$Hol(\nabla^{S})$ if and only if $\nabla^{S}_{Z}{A}=0$. We remind that $d(g^{L}(Z,\cdot))|_{\mathcal{L}^{\perp}}$ induces the Euler class of
$\mathcal{L}^{\perp}$ if $(X,g^{L},V,S)$ is almost horizontal.
\begin{lemma}
	Let $(X,g^{L},V)$ be an almost decent spacetime and $S$ a realization of the screen bundle. If $J \in \Gamma(X,O(S))$ with $J^{2}=-id_{S}$
	then $\nabla^{S}J=0$ if and only if $\nabla^{S}|_{\mathcal{L}^{\perp}}{J|_{\mathcal{L}^{\perp}}}=0$ for any leaf $\mathcal{L}^{\perp}$ of
	$\mathcal{X}^{\perp}$ and
	\begin{align*}
		0 &= d(g^{L}(Z,\cdot))(JY_{1},Y_{2}) + d(g^{L}(Z,\cdot))(Y_{1},JY_{2})\\
		 	&\qquad + g^{L}((L_{Z}{J})(Y_{1}),Y_{2}) - g^{L}((L_{Z}{J})(Y_{2}),Y_{1}).
	\end{align*}
\end{lemma}
\begin{proof}
	Define the extension $J \in \Gamma(X,End(TX))$ by $J(V)=J(Z)=0$ and let $\omega(\cdot,\cdot):=g^{L}(J(\cdot),\cdot) \in \Lambda^{2}T^{*}X$.
	Since $(L_{Z}{J})(Y) = [Z,JY] - J([Z,Y])$ we compute for $Y_{1},Y_{2} \in \Gamma(U,S)$
	\begin{align*}
		g^{L}((\nabla^{S}_{Z}{J})(Y_{1}),Y_{2}) &= g^{L}(\nabla^{S}_{Z}{(JY_{1})},Y_{2}) -g^{L}(J\nabla^{S}_{Z}{Y_{1}},Y_{2})\\
				&= g^{L}([Z,JY_{1}],Y_{2})+g^{L}(\nabla^{L}_{JY_{1}}{Z},Y_{2})\\
				&\qquad + g^{L}([Z,Y_{1}],JY_{2}) + g^{L}(\nabla^{L}_{Y_{1}}{Z},JY_{2})\\
				&= g^{L}((L_{Z}{J})(Y_{1}),Y_{2}) +g^{L}(\nabla^{L}_{JY_{1}}{Z},Y_{2})
				+ g^{L}(\nabla^{L}_{Y_{1}}{Z},JY_{2}),~\text{i.e.},
	\end{align*}
	$g^{L}((\nabla^{S}_{Z}{J})(Y_{1}),Y_{2}) - g^{L}((\nabla^{S}_{Z}{J})(Y_{2}),Y_{1}) = g^{L}((L_{Z}{J})(Y_{1}),Y_{2})
			- g^{L}((L_{Z}{J})(Y_{2}),Y_{1})+ d(g^{L}(Z,\cdot))(JY_{1},Y_{2}) + d(g^{L}(Z,\cdot))(Y_{1},JY_{2})$.
	We conclude the statement since $\nabla^{S}_{\cdot}{\omega}$ is a 2-form on $S$ and
	$\nabla^{S}_{Z}{\omega}(Y_{1},Y_{2}) = g^{L}((\nabla^{S}_{Z}{J})(Y_{1}),Y_{2})$.
\end{proof}
In order to estimate the higher Betti numbers we have to use the dual basic cohomology in the Gysin sequence of the flow if the basic cohomology does
not satisfy Poincar{\'e} duality. This is the case if and only if the Riemannian foliation $(\mathcal{L}^{\perp},\mathcal{X}|_{\mathcal{L}^{\perp}})$ is
not taut \cite{habib-richardson-2010}. Here we say $(\mathcal{L}^{\perp},\mathcal{X}|_{\mathcal{L}^{\perp}})$ is taut if
$H^{\dim \mathcal{L}^{\perp} -1}_{B}(\mathcal{X}|_{\mathcal{L}^{\perp}}) \neq 0$ which is equivalent to the vanishing of the {\'A}lvarez-class
$[\kappa_{\tilde{g}^{B}}] \in H^{1}_{B}(\mathcal{X}|_{\mathcal{L}^{\perp}})$.\par
By Cor. \ref{transverse-levi-civita} and \cite[Prop. 1.6]{MR0370617} the condition $\nabla^{S}|_{\mathcal{L}^{\perp}}{J|_{\mathcal{L}^{\perp}}}=0$ means that $J|_{\mathcal{L}^{\perp}}$ induces a K\"{a}hler foliation on $(\mathcal{L}^{\perp},\mathcal{X}|_{\mathcal{L}^{\perp}},g^{R}|_{\mathcal{L}^{\perp}})$.
In particular, basic Dolbeault cohomology is defined on $(\mathcal{L}^{\perp},\mathcal{X}|_{\mathcal{L}^{\perp}})$ \cite{MR1146730}. Suppose that
$\mathcal{L}^{\perp}$ is compact such that $Ric^{L}(W,W) \geq 0$ for all $W \in T\mathcal{L}^{\perp}$ and
$Hol(\nabla^{S}|_{\mathcal{L}^{\perp}})$ is irreducible. As in Prop. \ref{cohom-of-a-leaf} we conclude
$\dim H^{1}_{B}(\mathcal{X}|_{\mathcal{L}^{\perp}})=0$ since there is no $Hol(\nabla^{S}|_{\mathcal{L}^{\perp}})$-invariant vector. Hence,
$(\mathcal{L}^{\perp},\mathcal{X}|_{\mathcal{L}^{\perp}})$ is taut.
\begin{lemma}\label{transverse-calabi-yau-betti}
	Let $(X,g^{L},V)$ be an almost decent spacetime and $\mathcal{L}^{\perp}$ a compact leaf of $\mathcal{X}^{\perp}$. If
	$Ric^{L}|_{T\mathcal{L}^{\perp} \times T\mathcal{L}^{\perp}}=0$ and $Hol(\nabla^{S}|_{\mathcal{L}^{\perp}}) \subset U(n)$ then
	any basic $(p,0)$-form $\psi$ on $(\mathcal{L}^{\perp},\mathcal{X}|_{\mathcal{L}^{\perp}})$ is closed if and only
	if $\nabla^{S}|_{\mathcal{L}^{\perp}}{\psi}=0$.
\end{lemma}
\begin{proof}
	One part of the proof is implied by $d\psi = \sum_{i=1}^{\dim S}{e^{i} \wedge \nabla^{S}_{e_{i}}\psi}$.\par
	Since $\psi$ is a $(p,0)$-form we have $\bar{\delta}_{b}\psi=0$ and $d\psi=0$ implies $\bar{\partial}\psi=0$, i.e.,
	$\Delta_{b}^{\bar{\partial}}\psi = 0$. Thus, $\Delta_{b}\psi=0$ by the transverse K\"{a}hler identities \cite{MR1146730} and we have
	$d\psi =\delta_{b}\psi =0$ implying $\int_{\mathcal{L}^{\perp}}{\skal{\Delta_{\kappa}\psi}{\psi}}=\frac{1}{4}|\kappa|^{2}|\psi|^{2}$. By the
	Weitzenb\"{o}ck formula
	\begin{equation*}
		0 = \int_{\mathcal{L}^{\perp}}{|\nabla^{T}\psi|^{2}}
				+ \int_{\mathcal{L}^{\perp}}{\skal{\sum\nolimits_{i,j}{e^{j} \wedge e_{i} \lrcorner R^{T}(e_{i},e_{j})\psi}}{\psi}}.
	\end{equation*}
	However, $R^{T}$ beeing the curvature of $\nabla^{S}|_{\mathcal{L}^{\perp}}$ has the same symmetries as the curvature tensor of a K\"{a}hler
	manifold. Using the computation in \cite[Prop. 6.2.4]{MR1787733} we conclude
	$\sum_{i,j}{e^{j} \wedge e_{i} \lrcorner R^{T}(e_{i},e_{j})\psi}=0$, i.e., $0 = \int_{\mathcal{L}^{\perp}}{|\nabla^{T}\psi|^{2}}$.
\end{proof}
\begin{proposition}
	Let $(X,g^{L},V)$ be a decent spacetime and $\mathcal{L}^{\perp}$ a compact leaf of $\mathcal{X}^{\perp}$. Suppose
	$Hol(\nabla^{S}|_{\mathcal{L}^{\perp}})$ is irreducible and $Ric^{L}(W,W) \geq 0$ for all $W \in T\mathcal{L}^{\perp}$.
	\begin{enumerate}
		\item
		If $X$ is compact then $b_{1}(X) \in \{1,2\}$ and $b_{2}(X) \leq \dim H^{2}_{B}(\mathcal{X}|_{\mathcal{L}^{\perp}})+1$.
		\item
		If $X$ is non-compact and all leaves of $\mathcal{X}^{\perp}$ are compact then $b_{1}(X) \in \{0,1\}$ and
		$b_{2}(X) \in \{\dim H^{2}_{B}(\mathcal{X}|_{\mathcal{L}^{\perp}})-1,\dim H^{2}_{B}(\mathcal{X}|_{\mathcal{L}^{\perp}})\}$.
	\end{enumerate}
	Moreover, if $Hol(\nabla^{S}|_{\mathcal{L}^{\perp}})=SU(n)$ with $n \geq 3$ we can replace $H^{2}_{B}(\mathcal{X}|_{\mathcal{L}^{\perp}})$
	by $H^{1,1}_{B}(\mathcal{X}|_{\mathcal{L}^{\perp}})$ and if $Hol(\nabla^{S}|_{\mathcal{L}^{\perp}})=Sp(n)$ with $n \geq 1$ we can replace
	$\dim H^{2}_{B}(\mathcal{X}|_{\mathcal{L}^{\perp}})$ by $\dim H^{1,1}_{B}(\mathcal{X}|_{\mathcal{L}^{\perp}}) + 2$.
\end{proposition}
\begin{proof}
	Since $\dim H^{1}_{B}(\mathcal{X}|_{\mathcal{L}^{\perp}})=0$ and $(\mathcal{L}^{\perp},\mathcal{X}|_{\mathcal{L}^{\perp}})$ is taut we derive
	the bounds for $b_{1}(X)$ and the Gysin sequence implies
	$\R \stackrel{[\cdot \wedge \mathbf{e}]}{\longrightarrow} H^{2}_{B}(\mathcal{X}|_{\mathcal{L}^{\perp}})
			\longrightarrow H^{2}(\mathcal{L}^{\perp},\R) \longrightarrow 0$,
	i.e., $b_{2}(\mathcal{L}^{\perp})\in \{\dim H^{2}_{B}(\mathcal{X}|_{\mathcal{L}^{\perp}})-1,\dim H^{2}_{B}(\mathcal{X}|_{\mathcal{L}^{\perp}})\}$.
	If $X$ is compact the Mayer-Vietoris argument implies
	$H_{2}(\mathcal{L}^{\perp}) \stackrel{Id -F^{2}_{*}}{\longrightarrow} H_{2}(\mathcal{L}^{\perp})
				\stackrel{\iota_{*}}{\longrightarrow} H_{2}(X) \longrightarrow H_{1}(\mathcal{L}^{\perp})
					\stackrel{Id -F^{1}_{*}}{\longrightarrow} H_{1}(\mathcal{L}^{\perp})$.
	Hence, $b_{2}(X) = b_{1}(\mathcal{L}^{\perp}) + \dim \mbox{Eig}_{1}(F^{2}_{*})$, where $\mbox{Eig}_{1}(F^{2}_{*})$ is the eigenspace of
	$F^{2}_{*}$ w.r.t. the eigenvalue $1$, i.e.,
	$b_{2}(X) \leq b_{1}(\mathcal{L}^{\perp}) + b_{2}(\mathcal{L}^{\perp}) \leq \dim H^{2}_{B}(\mathcal{X}|_{\mathcal{L}^{\perp}}) +1$.
	The last statements follow from Lemma \ref{transverse-calabi-yau-betti}.
\end{proof}
\begin{proposition}
	Let $(X,g^{L},V,S)$ be a horizontal spacetime and $\mathcal{L}^{\perp}$ a leaf of $\mathcal{X}^{\perp}$. Suppose
	$d(g^{L}(Z,\cdot))|_{\mathcal{L}^{\perp}} \in \Lambda^{1,1}_{B}\mathcal{X}|_{\mathcal{L}^{\perp}}$ for some $\nabla^{S}$-parallel almost
	Hermitian structure $J$ on $S$. If $Z \in \Gamma(X,TX)$ is complete then there exists a complex structure on the universal cover of $X$.
\end{proposition}
\begin{proof}
	By Cor. \ref{all-leaves-closed} the universal cover of $X$ is diffeomorphic to $\tilde{X}:=\mathcal{L}^{\perp} \times \R^{+}$. We write $r$
	for the coordinate on $\R^{+}$ and $\eta:=g^{L}(Z,\cdot)|_{T\mathcal{L}^{\perp}}$. If $\Phi \in \mbox{End}(T\mathcal{L}^{\perp})$ is given by
	$\Phi(w \in S_{p}):=J(w)$ and $\Phi(V):=0$ then $(V,\eta,\Phi,g^{R}|_{\mathcal{L}^{\perp}})$ defines an almost contact metric structure on
	$\mathcal{L}^{\perp}$. On $\tilde{X}$ we define the cone metric $g^{C}:=dr^{r} + r^{2}g^{R}|_{\mathcal{L}^{\perp}}$ and the section
	$I \in \mbox{End}(T\tilde{X})$ by
	\begin{equation*}
		IY := \begin{cases}
			JY &\text{if}~Y \in S_{p},\\
			r\partial_{r} &\text{if}~Y=V,\\
			-V &\text{if}~Y=r\partial_{r}.
		      \end{cases}
	\end{equation*}
	Hence, we derive an almost Hermitian manifold $(\tilde{X},I,g^{C})$. By \cite[Thm. 6.5.9]{MR2382957} $I$ is integrable once we
	prove\footnote{In contrast to \cite{MR2382957} we define $d\eta(Y_{1},Y_{2}):= Y_{1}\eta(Y_{2})-Y_{2}\eta(Y_{1})-\eta([Y_{1},Y_{2}])$.} that
	$N_{\Phi}=-V \otimes d\eta$ where
	$N_{\Phi}(Y_{1},Y_{2}):=[\Phi Y_{1},\Phi Y_{2}] +\Phi^{2}[Y_{1},Y_{2}] - \Phi[Y_{1},\Phi Y_{2}] -\Phi[\Phi Y_{1},Y_{2}]$ for
	$Y_{\cdot} \in T\mathcal{L}^{\perp}$. Let $Y_{1} \in S$ and $Y_{2}=V$. Since $(X,g^{L},V,S)$ is horizontal we have
	$[Y_{1},V]=-\nabla^{S}_{V}{Y_{1}}$. Thus, $\Phi V =0$ and $J \circ \nabla^{S} = \nabla^{S}\circ J$ implies
	\begin{align*}
		N_{\Phi}(Y_{1},V) &= \Phi^{2}[Y_{1},V] - \Phi[\Phi Y_{1},V] = -\Phi^{2}(\nabla^{S}_{V}{Y_{1}}) + \Phi(\nabla^{S}_{V}{\Phi Y_{1}})\\
				&=-J^{2}(\nabla^{S}_{V}{Y_{1}}) + J(\nabla^{S}_{V}{JY_{1}}) =0.
	\end{align*}
	The same way we compute $g^{L}(N_{\Phi}(Y_{1},Y_{2}),Y_{3})=0$ if $Y_{1},Y_{2},Y_{3} \in S$. For $Y_{1},Y_{2} \in S$ we have
	\begin{align*}
		g^{L}(N_{\Phi}(Y_{1},Y_{2}),Z) &= g^{L}([\Phi Y_{1},\Phi Y_{2}],Z) = g^{L}([JY_{1},JY_{2}],Z)\\
					&= -g^{L}(\nabla^{L}_{JY_{1}}{Z},JY_{2}) + g^{L}(\nabla^{L}_{JY_{2}}{Z},JY_{1})\\
					&=-d(g^{L}(Z,\cdot))|_{\mathcal{L}^{\perp}}(JY_{1},JY_{2}).
	\end{align*}
	We conclude $g^{L}(N_{\Phi}(Y_{1},Y_{2}),Z) = -d\eta(Y_{1},Y_{2})$ since
	$d(g^{L}(Z,\cdot))|_{\mathcal{L}^{\perp}} \in \Lambda^{1,1}_{B}\mathcal{X}|_{\mathcal{L}^{\perp}}$ and
	$d\eta = d(g^{L}(Z,\cdot))|_{\mathcal{L}^{\perp}}$.
\end{proof}
%
%
%
%
\bibliographystyle{amsalpha}
\nocite{*}
\bibliography{topol-pp-wave}
\end{document}